\documentclass{siamltex}
\usepackage{amsfonts,latexsym}
\usepackage{amsmath}
\usepackage{subfigure}
\usepackage{amssymb}
\usepackage{graphicx}
\usepackage{amsmath}
\usepackage{color}
\usepackage{ulem}
\usepackage[top=3cm,bottom=4cm,right=4cm,left=4cm]{geometry}
\usepackage{booktabs} 
\usepackage{textcomp}
\usepackage{hyperref}
\usepackage{array}
\usepackage{adjustbox}
\newtheorem{df}{Definition}
\newtheorem{remark}{Remark}

\newtheorem{theorem1}{Theorem}
\newcommand{\R}{\mathbb R}

\newcommand{\cV}{\mathcal{V}}
\newcommand{\ctV}{\widetilde{\cV}}
\newcommand{\cW}{\mathcal{W}}

\newcommand{\cP}{\mathcal{P}}

\newcommand{\cG}{\mathcal{G}}

\newcommand{\cM}{\mathcal{M}}

\newcommand{\cR}{\mathcal{R}}

\newcommand{\cX}{\mathcal{X}}
\newcommand{\cY}{\mathcal{Y}}
\newcommand{\cO}{\mathcal{O}}

\newcommand{\cK}{{\cal K}}
\newcommand{\cZ}{{\cal Z}}

\newcommand{\cH}{{\cal H}}
\newcommand{\bT}{\mathbb {T}}
\newcommand{\cT}{{\cal T}}

\newcommand{\cL}{{\cal L}}

\newcommand{\bM}{\mathbb {M}}

\newcommand{\cU}{{\cal U}}
\newcommand{\cC}{{\cal C}}
\newcommand{\cI}{{\cal I}}
\newcommand{\cB}{{\cal B}}
\newcommand{\cF}{{\cal F}}
\newcommand{\cD}{{\cal D}}
\newcommand{\cE}{{\cal E}}
\newcommand{\cA}{{\cal A}}
\newcommand{\cQ}{{\cal Q}}

\newcommand{\cJ}{{\cal J}}

\newcommand{\IJJ}{I_1 \times \ldots \times I_M \times J_1 \times \ldots \times J_N}
\newcommand{\JJ}{J_1 \times \ldots \times J_N \times J_1 \times \ldots \times J_N}
\newcommand{\JJF}{J_1 \times J_2 \times J_1 \times J_2}
\newcommand{\JKF}{J_1 \times J_2 \times K_1 \times K_2}
\newcommand{\JKm}{J_1 \times J_2 \times K_1 \times mK_2}
\newcommand{\JKmm}{J_1 \times J_2 \times K_1 \times (m+1)K_2}
\newcommand{\KKm}{K_1 \times mK_2 \times K_1 \times mK_2}
\newcommand{\KTKm}{K_1 \times 2mK_2 \times K_1 \times 2mK_2}

\newcommand{\KK}{K_1 \times K_2 \times K_1 \times K_2}
\newcommand{\JK}{J_1 \times \ldots \times J_N \times K_1 \times \ldots \times K_M}
\newcommand{\JI}{J_1 \times \ldots \times J_N \times I_1 \times \ldots \times I_N}
\newcommand{\KKK}{K_1 \times \ldots \times K_M \times K_1 \times \ldots \times K_M}

\usepackage{epstopdf}
\usepackage{algorithm,algorithmic}

\title{A model reduction method for large-scale linear multidimensional dynamical systems}
\author{M.A. Hamadi \footnotemark[2] \thanks{Universit\'e Mohammed VI polytechnique, Lot 660, Hay Moulay Rachid, Ben Guerir, 43150 Maroc; amine.hamadi@um6p.ma; ahmed.ratnani@um6p.ma } \and K. Jbilou\footnotemark[1] \thanks{Universit\'e du Littoral C\^ote d'Opale, 50 Rue F. Buisson, BP 699--62228 Calais cedex, France;  jbilou@univ-littoral.fr} \and  A. Ratnani\footnotemark[1]}

\begin{document}
\maketitle
\begin{abstract}
%In this paper, we discuss the use of tensor algebra for multi-linear time invariant (MLTI) systems order reduction, as well as for the solution of some tensor equations. The key tool is the use of tensor Krylov subspace method. The latter is based on the idea of approximating the solution of the system within a low-dimensional subspace, known as the Krylov subspace, generated by applying a few tensor-tensor products. We introduce the classic tensor Krylov subspace via Einstein product for the block and global processes. An extension to the tensor structure of the extended Krylov subspace defined for the matrix case is established. Based on this methods, we give a model reduction process via projection technique, by using the transfer function that describes the input-output behaviour of the MLTI system. Moreover, those Krylov subspace projection methods are of a great tool to solve tensor equations. We show how to extract approximate solutions to a discrete Lyapunov tensor equations via Einstein product by means of the Arnoldi block and global processes using the classic and extended tensor Krylov subspaces.
In this work, we explore the application of multilinear  algebra in reducing the order of multidimentional linear time invariant (MLTI) systems. We use tensor Krylov subspace methods as  key tools, which involve approximating the system solution within a low-dimensional subspace. We introduce  the tensor extended block and global Krylov subspaces and the corresponding Arnoldi based processes. Using these methods, we develop a model reduction using projection techniques. We also show how these methods could be used to solve large-scale Lyapunov tensor equations that are needed in the balanced truncation method which is a technique for order reduction. We demonstrate how to extract approximate solutions  via the Einstein product using the tensor extended block Arnoldi and the extended global Arnoldi processes.
\end{abstract}
% Finally, section 3.4 sums-up the proposed approach
  
In model reduction for high multi-linear time invariant (MLTI) systems, multilinear  algebra can be used to simplify the system representation and reduce the computational complexity. This is achieved by decomposing the system tensor into smaller, lower-dimensional tensors, which can be more efficiently processed. One can use tensor decomposition techniques, such as the Tucker, the higher-order singular value decomposition (HOSVD) or  tensor train decomposition (TTD) to decompose the system into simpler components, which can then be used to identify and reduce the number of parameters needed to describe the system. 

Tensors are multi-dimensional arrays that can be used to represent higher-order relationships between variables in a system, making them a powerful tool for system analysis and reduction. The goal of model reduction in Multi-Linear Time Invariant  (MLTI) systems is to find a reduced-order representation that accurately captures the behavior of the original system, while reducing the computational cost. The resulting reduced-order model retains the essential features of the original system and can be used for tasks such as simulation, optimization and control. The use of multilinear  algebra in this context can significantly improve the efficiency and accuracy of the model reduction process. In a various scientific and engineering  applications, the dynamical model over space and time can be described by a systems of PDEs, {\it e.g.,} Navies-Stokes equations in fluid dynamic, Schrodinguer equations in quantum mechanics and heat equations in thermal transfer. Naturally, a spatial/temporal discretization to such equations yields to tensors that evolve in time. A straightforward idea that one can apply, is to transform tensors to matrices or vectors, and then use some linear time invariant approaches as, projection techniques onto Krylov subspaces \cite{houda1,frangos,hamadi,Druskin-on-optimal}, Balanced Truncation (BT) \cite{Antoulas3,stykel,moore} or   use  some non-intrusive methods, such as the Loewner framework, Eigensystem Realization Algorithm (ERA) \cite{era} and Dynamic Mode Decomposition (DMD) \cite{DMD}. The effectiveness of those methods have been clearly proved in a large body of literature. Other materials and details about model order reduction  techniques can be found in \cite{antoulas,guger2,benbook} and the references therein. However the transformation from tensors to vectors leads to  a computational challenge due to high number of states/model parameters of the original system that scales exponentially with the number of dimensions. The use of tensors has become increasingly significant in many control applications, such as state-space representation, system identification and optimal control, where tensors can be used to represent the state of a control system in a compact and efficient manner, making it easier to analyse the system and design control strategies. The main idea is to use multilinear algebra to tensorize the vector-based dynamic system representation into an equivalent tensor representation. This strategy of guarding the compact tensor based representations/computations have been established in many applications, such as linear parameters varying models \cite{LPV}, nonlinear circuits via tensor decomposition \cite{nlcircuit}, solution of Lyapunov equations \cite{Nip} and in simulation of MLTI systems \cite{MLTISim}.

In this paper we are interested in projecting the MLTI system onto some Krylov subspaces of a finite dimension using Einstein product. This technique of projection helps to reduce the complexity of the original system while preserving its essential features, and what we end up with is a lower-dimensional reduced MLTI system that has the same structure as the original one. The MLTI systems have been introduced in \cite{CAN1}, where the states, outputs and inputs are preserved in a tensor format and the dynamic is supposed to be described by some multilinear operators. A technique named tensor unfolding \cite{Brazell} which transforms tensor to matrix, allows to extend many concepts and notions in the LTI theory to the tensor case. One of them is the use of the transfer function that describes the input-output behaviour of the MLTI system, and to quantify the efficiency of the projected reduced MLTI system. Many properties in the theory of LTI systems such as, reachibility, controllability and stability, have been generalized in the case of MLTI systems based via the Einstein product \cite{CAN2}. A variety of applications in model reduction and identification from tensor time series data have been established in \cite{tensortimeserie,tensortimeserie2}. The authors have shown the efficiency behind the compact representation based on tensor algebra. A Balanced Truncation (BT) model reduction for MLTI systems based on tensor structure has been introduced in \cite{CAN2}, where authors employed the tensor train decomposition (TTD) to make computation more feasible. Through the paper, we will be interested in projection  and balanced truncation techniques using  Krylov tensor based subspaces, namely tensor extended global and block Krylov subspace methods using new  Arnoldi-based process. This will be done using the  Einstein product and multilinear algebra  to decompose the system into simpler components via projection.  In the matrix case, the efficiency of the extended Krylov susbpace method has been shown and proved in many fields, in LTI model reduction \cite{Abidi,houda1}, also for solving large scale algebraic matrix equations \cite{hey,heyouni0,simoncini,simon}. We introduce tensor extended  block and global Krylov subspaces and we use them via an Arnoldi process in model reduction for MLTI, also for solving some tensor equations based on Einstein product.

Recently, tensor Krylov subspace methods have gained interest beyond scientific community. These methods are based on the idea of approximating the solution of the system within a low-dimensional subspace, known as the Krylov subspace, generated by applying a few tensor-tensor products. Iterative Krylov subspace methods in tensor structure have shown efficiency and accurate solution of high-dimensional problem. They can be used to solve large multi-linear tensor equations \cite{AlaaGuide,AlaaGuide2} with  applications to color images and video processing, using Einstein and T-products. These methods exploit the fact that the computational cost grows linearly with the size of the problem. Another application to solve multilinear systems based on Einstein product via tensor inversion has been achieved in \cite{Brazell}. In a second part of the paper, we will be interested in some large scale tensor equations such as  the discrete Lyapunov tensor equations  that appears when using  balanced truncation model order reduction methods. We show how to extract approximate solution using tensor-based Krylov subspace methods. 

The paper is organized as follows. In Section \ref{prelem}, we present  general notations and give the main definitions  used throughout the paper. An introduction to the MLTI systems along with a description to the tensor classic and extended global  Krylov subspace based to construct high fidelity reduced MLTI systems is shown in Section \ref{mlti}. The reduction process via projection using tensor Krylov subspaces is established in Section \ref{secreduproj}. A description to the tensor balanced truncation method is presented in Section \ref{BTsec}. Computing approximate solution to large scale discrete Lyapunov tensor equation via the tensor block and global Krylov subspaces is described in Section \ref{disclyap}. Finally, in Section \ref{numerexamp}, some numerical experiments are assessed to show the accuracy of the proposed methods.

\section{Background and definitions}
\label{prelem}
In this section, we provide with some notations and definitions that will be used throughout this paper. Unless otherwise specified, we represent tensors by calligraphic letters, we use lowercase and capital letters to identify vectors and matrices, respectively. An  $N$-mode tensor  is a multidimensional array  and can be represented by $\cX \in\R^{J_1 \times\ldots\times J_N}$, whose elements  are denoted  by $x_{j_1,\ldots,j_N}$ with $1\leq j_i \leq J_i$, $i=1,\ldots, N$. If $N=0$, then $\cX \in \R$ is a $0$th order tensor that we call a scalar, a vector is a $1$th order tensor and a matrix is $2$th order tensor. 
Further definitions and explanations about tensors could be found in \cite{Kolda1}. The tensor Einstein product is a  multidimensional generalizations of matrix product;  for more details see \cite{Brazell,CAN2}
\begin{df}
	Let  $\cA \in \R^{\JK}$ and in $\cB \in \R^{K_1 \times \ldots \times K_M \times I_1 \times \ldots \times I_L}$  be two tensors. The Einstein product $\cA \ast_M \cB \in \R^{J_1 \times \ldots \times J_N\times I_1 \times \ldots \times I_L}$ is defined as follows
	\begin{equation}
		(\cA \ast_M \cB)_{j_1\ldots j_Ni_1\ldots i_L} = \sum_{k_1=1}^{K_1} \sum_{k_2=1}^{K_2} \ldots \sum_{k_M=1}^{K_M} a_{j_1\ldots j_Nk_1\ldots k_M}b_{k_1\ldots k_Mi_1\ldots i_L}.
	\end{equation}
\end{df}
Some special tensor that will be considered in this paper : A  nonzero tensor $\cD=(d_{j_1\ldots j_Ni_1\ldots i_N}) \in \R^{\JJ}$, is called a {\it Diagonal tensor} if all the   entries $d_{j_1\ldots j_Ni_1\ldots i_N}$ are equal to zero expect the diagonal entries denoted by $d_{j_1\ldots j_Nj_1\ldots j_N} \neq 0$. In the case where all the diagonal entries are equal to $1$ ({\it i.e., $d_{j_1\ldots j_Nj_1\ldots j_N}=1$}) then $\cD$ is called the {\it identity} tensor denoted by $\cI$.  Let $\cA \in \R^{\JK}$ and $\cB \in \R^{ K_1 \times \ldots \times K_M\times J_1 \times \ldots \times J_N}$ two tensors verified $b_{k_1\ldots k_Mj_1\ldots j_N}=a_{j_1\ldots j_Nk_1\ldots k_M}$, then $\cB$ is called the {\it transpose} of $\cA$ and denoted by $\cA^T$. The inverse of a square tensor is defined as follows.
\begin{df}
	A square tensor  $\cA  \in \R^{\JJ}$ is invertible iff there exists a tensor $\cX  \in \R^{\JJ}$ such that 
	$$\cA \ast_N \cX = \cX \ast_N \cA = \cI \in \R^{\JJ}.$$
	In that case, $\cX$  is the inverse of $\cA$ and is denoted by $\cA^{-1}.$ 
\end{df}
\begin{proposition}
	Consider $\cA\in \R^{\JK}$ and $\cB\in \small \R^{K_1 \times \ldots \times K_M\times J_1 \times \ldots \times J_N}$. Then we have the following relations 
	\begin{enumerate}
		\item $(\cA \ast_M \cB)^T = \cB^T \ast_N \cA^T.$
		\item $\cI_M \ast_M \cB = \cB$ and $\cB \ast_N \cI_N = \cB$.
	\end{enumerate}
\end{proposition}
\medskip 
\noindent The {\it trace} denoted by $tr(\cdot)$ of a square-order tensor $\cA\in \R^{\JJ}$ is given by 
\begin{equation}
	\label{traceten}
	tr(\cA)= \sum_{j_1\ldots j_N} a_{j_1\ldots j_Nj_1\ldots j_N}.
\end{equation}
The inner product of the two tensors $\cX, \cY \in \R^{\JK}$ is defined as 
$$<\cX, \cY>=tr(\cX^T\ast_N \cY),$$
where $\cY^T \in \R^{K_1 \times \ldots \times K_M\times J_1 \times \ldots \times J_N}$ is the transpose tensor of $\cY$. If $<\cX, \cY>=0$, we say that $\cX$ and $\cY$ are orthogonal. The corresponding norm is the tensor Frobenius norm given by 
\begin{equation}
	\label{tennorm}
	\|\cX\| = \sqrt{tr(\cX^T \ast_N \cX)}.
\end{equation}

We recall the definition of the eigenvalues of a tensor mentioned in \cite{eigenten,eigeneventen}.
\begin{df}
	\label{defteneigen}
	Let $\cA \in \R^{\JJ}$ be a tensor. The complex scalar $\lambda$ satisfied 
	$$\cA \ast_N \cX = \lambda\cX,$$
	is called an eigenvalue of $\cA$, where $\cX \in \R^{J_1 \times\ldots\times J_N}$ is a nonzero tensor named an eigentensor. The set of all the eigenvalues of $\cA$ is denoted by $\Lambda(\cA)$.
\end{df}
\subsection{Tensor unfolding}
Tensor unfolding plays an important role in tensor computations. It provides simplification that allows to extended many notions to the tensor case that are already presented in the matrix case. We reorganize the entries of the tensor and present them in a vector or in a matrix format. 
\begin{df}
	\label{transphi}
	Consider the following transformation $\Psi : \bT_{j_1\ldots j_Nk_1\ldots k_M} \rightarrow \bM_{|J|,|K|}$ with $\Psi(\cA)=A$ defined component wise as 
	$$\cA_{j_1\ldots j_Nk_1\ldots k_M} \rightarrow A_{i\text{vec}(j,J),i\text{vec}(k,K)},$$
	where we refer to $J=\{J_1,\ldots,J_N\}$, $K=\{K_1,\ldots,K_M\}$ and  $\bT_{j_1\ldots j_Nk_1\ldots k_M}$ is the set of all tensors in $\R^{\JK}$. $ \bM_{|J|,|K|}$ is the set of matrices in $\R^{|J| \times |K|}$ where $|J|=J_1\ldots J_N$ and $|K|=K_1\ldots K_N$. The index mapping $i\text{vec}(\cdot,\cdot)$ introduced in \cite{Ragnarsson} is given by
	$$i\text{vec}(j,J)= j_1+\sum_{l=2}^{N}(j_l-1) \prod_{l=1}^{s-1}J_s.$$
	$$i\text{vec}(k,K)= k_1+\sum_{l=2}^{M}(k_l-1) \prod_{l=1}^{s-1}K_s.$$
\end{df}
\subsection{Block tensors}
\label{blockten}
In this subsection, and analogously to block matrices, we describe briefly how our tensors can be stored properly in a one tensor. To clarify our description we need to start  by giving  the definition of an {\it even-order paired tensor.}
\begin{df}
	A tensor is called even-order paired tensor if its order is always even, {\it i.e.,} ($2N$ for example), and the elements of the tensor are indicating using a pairwise index notation, {\it i.e.,} $a_{j_1i_1\ldots j_Ni_N}$ for $\cA\in \R^{J_1\times I_1 \times \ldots \times J_N\times I_N}$. 
\end{df}
A definition of {\it $n$-mode row block tensor}, of two even-order paired tensor of the same size is given by 
\begin{df}(\cite{CAN2})
	\label{rowblock}
	Consider $\cA,\cB \in \R^{J_1\times I_1 \times \ldots \times J_N\times I_N}$ two even-order paired tensors, the {\it $n$-mode row block tensor} denoted by   $\displaystyle \big\vert \cA \,\, \cB\big\vert_{n}\in \R^{J_1\times I_1 \times\ldots\times J_n\times 2I_n \times \ldots\times J_N\times I_N}$  is defined by the following  \\[0.2cm]
	\begin{equation}
		(\big\vert \cA\, \, \cB\big\vert_n)_{j_1\ldots j_N i_1\ldots i_N} = \left\{\begin{array}{lll}
			a_{j_1\ldots j_N i_1\ldots i_N}, & & j_k=1,\ldots,J_k,\, i_k=1,\ldots,I_k \, \forall k\\
			b_{j_1\ldots j_N i_1\ldots i_N}, & & j_k=1,\ldots,J_k,\, \forall k, \, i_k=1,\ldots,I_k \, \,\text{for}\, \, k\neq n.\\
			& & \text{and} \,\, i_k=I_k+1,\ldots, 2I_k\, \,\text{for}\, \,k=n.
		\end{array}\right.
	\end{equation}
\end{df}
In this work, we are interested in the $4th$-order tensors ({\it i.e., $N=2$}), and not especially paired ones ({\it i.e., the paired wise index notation is not used}). We  have noticed that the definition above stills valid for the case of $4th$-order paired or non-paired tensor. In this regard, we specify more in the following this block tensor notation by giving some examples. We suppose now that $\cA$ and $\cB$ are two $4th$-order tensors in the space $\R^{J_1 \times J_2 \times I_1\times I_2}$. By following the two definitions above, we refer to $1$-mode row block tensor by 
%\begin{itemize}
%	\item [] 
$\big\vert \cA \,\, \cB\big\vert_1 \in \R^{J_1\times 2J_2\times I_1\times I_2}$. We use the MATLAB colon operation $:$ to explain how to extract the two tensors $\cA$ and $\cB$ as follows
$$\big\vert \cA \,\, \cB\big\vert_1(:,1:J_2,:,:) = \cA \quad \text{and} \quad \big\vert \cA \,\, \cB\big\vert_1(:,J_2+1:2J_2,:,:) = \cB.$$
%\item [] 
The notation of $2$-mode row block tensor $\big\vert \cA \,\, \cB\big\vert_2 \in \R^{J_1\times J_2\times I_1\times 2I_2}$ described below is the mostly used in this paper. The extraction goes as 
$$\big\vert \cA \,\, \cB\big\vert_2(:,:,:,1:K_2) = \cA \quad \text{and} \quad \big\vert \cA \,\, \cB\big\vert_2(:,:,:,K_2+1:2K_2) = \cB.$$ 
%\end{itemize}
\begin{remark}
	We notice that if we consider $N=1$ ({\it i.e., $\cA$ and $\cB$ are now matrices, we refer to them as $A$ and $B$}), then the notation $\big\vert \cA \,\, \cB\big\vert_1=\big\vert \cA \,\, \cB\big\vert_2 =[A,B] \in \R^{J_1 \times 2I_1}$ is now defined by the standard block matrices definition.
\end{remark}
\noindent
Similarly to the $1$ or $2$-mode row block tensor, we can define the $1$ or $2$-mode column block tensor for the $4th$-order tensors based on the definition proposed by Can et al., in \cite{CAN1}.
Using Matlab notation, 	$\begin{vmatrix}
	\cA\\
	\cB
\end{vmatrix}_1$ is the  1-mode column block tensor in $\R^{2J_1\times J_2\times I_1\times I_2}$, defined by 
$$\begin{vmatrix}
	\cA\\
	\cB
\end{vmatrix}_1(1:J_1,:,:,:) = \cA \quad \text{and} \quad \begin{vmatrix}
	\cA\\
	\cB
\end{vmatrix}_1(J_1+1:2J_1,:,:,:) = \cB,$$
%\item[] 
and $\begin{vmatrix}
	\cA\\
	\cB
\end{vmatrix}_2$ is 2-mode column block tensor in $\R^{J_1\times J_2\times 2I_1\times I_2}$ given by 
$$\begin{vmatrix}
	\cA\\
	\cB
\end{vmatrix}_1(:,:,1:I_1,:) = \cA \quad \text{and} \quad \begin{vmatrix}
	\cA\\
	\cB
\end{vmatrix}_1(:,:,I_1+1:2I_1,:) = \cB.$$
%\end{itemize}
Based on the Einstein product,  we can easily prove the following proposition that will be mainly used during the computational process described in the following sections.
\begin{proposition}
	Consider four tensors of the same size $\cA,\cB,\cC,\cD \in \R^{K_1 \times K_2\times  K_1\times K_2}$, then we get
	{\small	\begin{equation*}
			\begin{array}{lll}
				\underbrace{\big\vert \cA \,\, \cB\big\vert_2}_{\in \R^{K_1\times K_2\times K_1\times 2K_2}} \ast_2 \underbrace{\big\vert \cC \,\, \cD\big\vert_1}_{\in \R^{K_1\times 2K_2\times K_1\times K_2}} &=&\cA\ast_2 \cC+ \cB\ast_2 \cD \in \R^{K_1 \times K_2\times K_1\times K_2}. \\[0.2cm]
				\underbrace{\big\vert \cA \,\, \cB\big\vert_1}_{\in \R^{K_1\times 2K_2\times K_1\times K_2}} \ast_2 \underbrace{\big\vert \cC \,\, \cD\big\vert_2}_{\in \R^{K_1\times K_2\times K_1\times 2K_2}} &=&
				\Big\vert \big\vert \cA\ast_2 \cC \,\,\, \cA\ast_2 \cD\big\vert_2 ,\,\,\big\vert \cB\ast_2 \cC \,\,\, \cB\ast_2 \cD\big\vert_2\Big\vert_1 \in \R^{K_1 \times 2K_2\times K_1\times 2K_2}.
			\end{array}
	\end{equation*}}
\end{proposition}
Let us consider that we have a $m$ tensors $\cA_k \in \R^{J_1 \times J_2 \times I_1\times I_2}$, we can apply the definition associated to the $2$-mode row block tensor in succession to create a one tensor denoted by $\cA$ in the space $\R^{J_1 \times J_2 \times I_1\times mI_2}$ and has the following form
\begin{equation}
	\label{tenblsucc}
	\cA =\big\vert \cA_1 \quad \cA_2 \quad \ldots \quad \cA_m\big\vert \in \R^{J_1 \times J_2 \times I_1\times mI_2}. 
\end{equation}
For more general definition we refer the reader to [\cite{CAN2}, Definition $4.2$]. The process to create such tensor is defined as follows
\begin{itemize}
	\item [-] If $m$ is {\it even} then we consider the following
	$$\cX_{12} = \underbrace{\big\vert \cA_1 \quad \cA_2\big\vert_2}_{\in \R^{J_1\times J_2\times I_1\times 2I_2}} , \quad \cX_{34} = \underbrace{\big\vert \cA_3 \quad \cA_4\big\vert_2}_{\in \R^{J_1\times J_2\times I_1\times 2I_2}}, \quad \ldots, \quad \cX_{m-1m} = \underbrace{\big\vert \cA_{m-1} \quad \cA_m\big\vert_2}_{\in \R^{J_1\times J_2\times I_1\times 2I_2}},$$
	then we stored them two by two as follows
	$$\cX_{1234} = \underbrace{\big\vert \cX_{12} \quad \cX_{34}\big\vert_2}_{\in \R^{J_1\times J_2\times I_1\times 4I_2}},  \quad \ldots, \quad \cX_{m-3m-2m-1m} = \underbrace{\big\vert \cX_{m-3m-2} \quad \cX_{m-1m}\big\vert_2}_{\in \R^{J_1\times J_2\times I_1\times 4I_2}},$$
	We keep repeating the process until the last 2-mode row block tensor obtained and simply we presented in the following notation
	\begin{equation}
		\cA =\big\vert \cA_1 \quad \cA_2 \quad \ldots \quad \cA_m\big\vert \in \R^{J_1 \times J_2 \times I_1\times mI_2}.
	\end{equation}  
	\item [-] If $m$ is {\it odd}, then we first choose $m-1$ tensors and applied the process described above. We denote by $\cX$ the obtained tensor in $\R^{J_1 \times J_2 \times I_1\times (m-1)I_2}$ and then we obtain the desired tensor $\cA_m$ given by 
	$$\cA = \big\vert \cX \quad \cA_m \big\vert\in \R^{J_1 \times J_2 \times I_1\times mI_2}.$$ 
	An explanation of the MATLAB colon $:$ for the constructed tensor $\cA$ goes as follows
	$$ \cA_l=\cA(:,:,:,(l-1)I_2+1:lI_2) \quad \text{for} \,\, l=1,\ldots,m.$$
\end{itemize}
In Sections \ref{mlti} and \ref{disclyap},  we will deal with some tensors in $\R^{\KKm}$. Let us denote this tensor $\cM_m$, and suppose that $m=3$ ({\it i.e., $\cM_3 \in \R^{K_1\times 3K_2\times K_1\times 3K_2}$}), by following the definitions above, the tensor $\cM_3$ can be presented as follows
\begin{equation}
	\cM_3 = \bigg\vert \,\Big\vert\,\big\vert\cM_{1,1} \quad\cM_{1,2}\big\vert_2\quad \cM_{1,3}\Big\vert_2,\,\, \Big\vert\, \big\vert\cM_{2,1} \quad\cM_{2,2}\big\vert_2\quad \cM_{2,3}\Big\vert_2,\, \Big\vert\, \big\vert \cM_{3,1} \quad\cM_{3,2}\Big\vert_2\quad \cM_{3,3}\Big\vert_2\,\,\bigg\vert_1,
\end{equation}
where $\cM_{i,j} \in \R^{\KK}$ and the three $2$-mode row block tensor given above have been constructed as mentioned in (\ref{tenblsucc}), where we concatenate every tensor $\cM_{i,j}$ successively. To facilitate the notation, we will refer to $\cM_m$ as 
\begin{equation}
	\cM_3=	\begin{vmatrix}
		\cM_{1,2}	& \cM_{1,2} & \cM_{1,3} \\
		\cM_{2,1}	&  \cM_{2,2}& \cM_{2,3} \\
		\cM_{3,1}	& \cM_{3,2} &\cM_{3,3} 
	\end{vmatrix} \in \R^{K_1\times 3K_2\times K_1\times 3K_2}.
\end{equation}
The tensors $\cM_{i,j}$ could be defined from the tensor $\cM_m$ by using the MATLAB colon as follows
$$ \cM_{i,j}=\cM_3(:,(i-1)K_2+1:iK_2,:,(j-1)K_2+1:jK_2) \quad \text{for} \,\, i,j=1,2,3.$$
This could be followed by a generalisation to $m \in \mathbb{N}$ ({\it i.e.,} $\cM_m \in \R^{K_1\times mK_2\times K_1\times mK_2}$) defined as follows
\begin{equation}
	\label{Hmtensor}
	\cM_m=\begin{vmatrix}
		\cM_{1,1}	& \cM_{1,2} & \ldots  & \cM_{1,m} \\
		\cM_{2,1}	& \cM_{2,2} &\ldots  &\cM_{2,m}  \\
		\vdots	&  \ddots&\ddots  &\vdots  \\
		\cM_{m,1}	& \cM_{m,2}  &\ldots  &\cM_{m,m} 
	\end{vmatrix} \in \R^{K_1\times mK_2\times K_1\times mK_2},
\end{equation}
and in the same way, we can describe every tensor $\cM_{i,j}$ using the MATLAB colon by the following 
$$\cM_{i,j} =\cM_m(:,(i-1)K_2+1:iK_2,:,(j-1)K_2+1:jK_2) \in \R^{\KK} \quad \text{for} \,\, i,j=1,2, \ldots, m.$$
For a simple presentation, we suppose that $m=2$ and $\cA,\cB$ are two tensors in $\R^{\KK}$, then using the definitions described above, we can easy prove the following proposition
\begin{proposition}
	\label{rowblocktimesrowblock}
	
	\begin{itemize}
		\item [1.] $\begin{vmatrix}
			\cM_{1,1}	& \cM_{1,2}\\
			\cM_{2,1}	&  \cM_{2,2}
		\end{vmatrix}=\Big\vert\,\big\vert\cM_{1,1} \quad\cM_{1,2}\big\vert_2\,\, \big\vert\cM_{2,1} \quad\cM_{2,2}\big\vert_2\Big\vert_1 = \Big\vert\,\big\vert\cM_{1,1} \quad\cM_{2,1}\big\vert_1\,\, \big\vert\cM_{1,2} \quad\cM_{2,2}\big\vert_1\Big\vert_2$.
		\item[2.] $\big\vert \cA \quad \cB \big\vert_2 \ast_2\begin{vmatrix}
			\cM_{1,1}	& \cM_{1,2}\\
			\cM_{2,1}	&  \cM_{2,2}
		\end{vmatrix} = \big\vert \cA\ast_2 \cM_{1,1} + \cB\ast_2 \cM_{2,1}  \quad \cA\ast_2 \cM_{1,2} + \cB\ast_2 \cM_{2,2}\big\vert_2.$
		\item  [3.]  $\begin{vmatrix}
			\cM_{1,1}	& \cM_{1,2}\\
			\cM_{2,1}	&  \cM_{2,2}
		\end{vmatrix} \ast_2 \big\vert \cA \quad \cB \big\vert_1 = \big\vert \cM_{1,1} \ast_2 \cA + \cM_{1,2} \ast_2 \cB\quad \cM_{2,1} \ast_2 \cA + \cM_{2,2}\ast_2 \cB \big\vert_1.$
	\end{itemize}
\end{proposition}
\section{Multidimensional linear model reduction}
\label{mlti}
Consider the following discrete time MLTI system
\begin{equation}
	\label{oryml}
	\left\{\begin{split} 
		\cX_{k+1} &= \cA\ast_N\cX_k+ \cB \ast_M \cU_k, \quad \cX_0=0,\\ 
		\cY_k &= \cC \ast_N \cX_k,
	\end{split} \right.
\end{equation}
in the continuous case, a continuous MLTI system could be described as follows
\begin{equation}
	\label{orymlc}
	\left\{\begin{split} 
		\dot{\cX}(t) &= \cA\ast_N\cX(t)+ \cB \ast_M \cU(t), \quad \cX_0=0,\\ 
		\cY(t) &= \cC \ast_N \cX(t), 
	\end{split} \right.
\end{equation}
where   $\dot{\cX}(t)$ is the derivative  of the tensor $\cX(t) \in \R^{J_1 \times\ldots\times J_N}$,  $\cA \in \R^{\JJ}$ is a  square  tensor, $\cB \in \R^{\JK}$ and $\cC \in \R^{\IJJ}$. The tensors $\cU_k \in \R^{K_1 \times \ldots \times K_M}$ ($\cU(t)$ {\it in the continuous case}) and $\cY_k\in \R^{I_1 \times \ldots \times I_M}$ ($\cY(t)$ {\it in the continuous case}) are the  control and output tensors, respectively.
\begin{proposition}
	\label{TFten}
	The  transfer function associated  to the dynamical system (\ref{orymlc}) is given by
	\begin{equation}
		\label{tf_mlti}
		\cF(s) = \cC \ast_N (s\cI-\cA)^{-1} \ast_N \cB.
	\end{equation}
\end{proposition}
\begin{proof}
	Given the transformation $\Psi$ defined in \ref{transphi}, we can show that the multi-linear system (\ref{oryml}) could be transformed into an equivalent matrix linear  dynamical system, defined by 
	\begin{equation}
		\label{linsys}
		\left\{\begin{split} 
			x_{k+1} &= Ax_k+ B  u_k, \quad x_0=0\\ 
			y_k &= Cx_k, 
		\end{split} \right.
	\end{equation}
	where $A,B$ and $C$ are the matrices resulting  from the matricization of the tensors $\cA, \cB$ and $\cC$ respectively by using $\Psi$. The same remark goes  for the vectors $x_k,u_k$ and $y_k$. System of the form (\ref{linsys}), can be represented in a frequency domain via Z-transform by the following transfer function (cf. \cite{discsys})
	$$F(s)=C(sI-A)^{-1}B.$$
	From [\cite{Brazell}, Lemma $3.1$], it is proved that for two matrices $X$ and $Y$ with appropriate dimensions, we have   $\Psi^{-1}(XY)=\cX \ast \cY$. Then, we get the following result
	\begin{equation*}
		\begin{split}
			\cF(s) =\Psi^{-1}(F(s)) &= \Psi^{-1}(C(sI-A)^{-1}B) \\%\text{is the associated transfer function to (\ref{linsys})}) 
			&=\Psi^{-1}(C) \ast_N \Psi^{-1}((sI-A)^{-1}) \ast_N \Psi^{-1}(B) \\
			&=\cC \ast_N (s\cI-\cA)^{-1} \ast_N \cB.
		\end{split}
	\end{equation*}
\end{proof}
It is worthy to mention that the theoretical  concepts that characterize the linear dynamical  system (\ref{linsys}) such as  stability, reachability and observability have been extended to the MLTI systems. Originally, these theoretic concepts have been introduced by Can et al., in \cite{CAN2} in the case of a discrete MLTI system (\ref{oryml}). The results shown below are exactly the ones described for the matrix  case if one considers the isomorphism $\Psi$ described previously.
If we consider a non-zero initial tensor state $\cX_0$, and  analogously to the matrix  case, the solution to the discrete MLTI systems (\ref{oryml}) is given by
% the solution to the continuous
\begin{equation}
	\begin{array}{lll}
		\cX_k&=& \cA^{\star k} \ast_N \cX_0 +\displaystyle\sum_{j=0}^{k-1} \cA^{\star k-j-1} \ast_N \cB \ast_M \cU_j, %\quad \text{(disc. case),} \\
		%		\cX_k&=& e^{\cA(t-t_0)}\ast\cX_0 + \displaystyle \int_{t_0}^{t}e^{\cA(t-\tau)}\ast \cB \ast \cU(\tau)\text{d}\tau. \quad \text{(cont. case).}
	\end{array}
\end{equation}
where $\cA^{\star k} =  \underbrace{\cA \ast \ldots \ast \cA}_{k}$. The notion of stability of an MLTI system is given as follows
\begin{df}
	Given the MLTI system (\ref{oryml}). The system is called 
	\begin{itemize}
		\item[1.]  {\it Asymptotically stable} if $\|\cX_k\|\leq \kappa \|\cX_0\| $ for some $\kappa >0$.
		\item[2.] {\it Stable} if $\|\cX_k\| \rightarrow 0$ as $t \rightarrow \infty$. 
	\end{itemize}
\end{df}
\medskip 
\noindent We have the following results related to the stability of an MLTI system.
\begin{proposition}(\cite{CAN1})
	\begin{itemize}
		\item[1.] System (\ref{oryml}) is considered {\it stable } if and only if  the magnitudes of all the eigenvalues of $\cA$ ({\it Definition \ref{defteneigen}}) are less than or equal to $1$. Additionally, for any eigenvalues equal to $1$, the algebraic and geometry multiplicity must be equal.
		\item[2.] If the magnitudes of all the eigenvalues of $\cA$ are less than $1$ than System (\ref{oryml}) is called asymptotically stable.
	\end{itemize}
\end{proposition}
\noindent
In the discrete case, the definitions of the reachability and the observability of a discrete MLTI system are defined below. 
\begin{df}
	Consider the discrete MLTI system (\ref{oryml}). The pair ($\cA,\cB$)   is reachable on $\left[k_0,k_1\right]$ if and only if the {\it reachability Gramian} 
	\begin{equation}
		\cP_d(k_0,k_1)  =  \sum_{k=k_0}^{k_1-1} \cA^{k_1-k-1} \ast_N\cB \ast_M \cB^T\ast_N (\cA^{k_1-k-1})^T,
	\end{equation}
	is a weakly symmetric positive-definite square tensor, see \cite{CAN1} for more details about this notion.
\end{df}
\noindent
The reachability Gramian is a solution of  the following discrete  Lyapunov (called also Stein) tensor equation
\begin{equation}
	\label{layptenre}
	\cP_d - \cA \ast_N \cP_d \ast_N \cA^T= \cB \ast_N \cB^T.
\end{equation} 
Likewise, the observability of an MLTI system is given as follows. 
\medskip
\begin{df}
	The pair ($\cA,\cC$) is observable on $\left[k_0, k_1\right]$ if and only if the observability Gramian
	\begin{equation}
		\cQ_d(k_0,k_1)  =  \sum_{k=k_0}^{k_1-1} (\cA^{k_1-k-1})^T \ast_N\cC^T \ast_M \cC\ast_N \cA^{k_1-k-1},
	\end{equation}
	is a weakly symmetric positive-definite square tensor, see \cite{CAN1} for more details about this notion.
\end{df} 

\noindent
This observability Gramian $\cQ_d$ determines how much each state influences future outputs \cite{Rowley}. It fulfils the following discrete {\it Discrete-Lyapunov tensor equation}
\begin{equation}
	\label{layptenob}
	\cQ_d - \cA^T \ast_N \cQ_d \ast_N \cA= \cC^T \ast_M \cC.
\end{equation} 

\noindent
The two discrete Lyapunov tensor equations (\ref{layptenre}) and (\ref{layptenob}) play a crucial role in the Balanced Truncation model order reduction for MLTI systems described in \cite{CAN3}. A technique based on TT-algebra \cite{CAN3} and outer product has been established to solve such equations. In the next section, we describe a tensor Krylov subspace projection technique for solving properly such  equations. As in the case of LTI systems, in order to construct a reduced model using a Krylov subspace techniques, the whole process remains on approximating the associated transfer function. This could also be shared with the MLTI systems. In the following, we introduce a tensor global  Arnoldi Krylov method using  Einstein product to construct such approximation that yields at the end to  the reduced MLTI model. We use the classic Krylov subspace and also introduce the tensor extended Krylov subspace.
\subsection{Tensor based global Krylov subspace methods}
\subsubsection{Classic  Krylov subspace processing}
The first step in the approximation of the associated transfer function (\ref{tf_mlti}) to the MLTI system (\ref{oryml}) is that we rewrite it as 
$$\cF(s) = \cC \ast_N \cX, $$
where $\cX$ verify the following multi-linear system
\begin{equation}
	\label{mlsys}
	(s\cI-\cA) \ast_N \cX = \cB.
\end{equation}
In order to find an approximation to $\cF(s)$, it remains to find an approximation to the above  multi-linear system, which can be done by using a projection Krylov subpace technique. Let us first define the $m$-th tensor Krylov subspace
\begin{equation}
	\label{t_ks}
	\cK_m(\cA, \cB) = \text{span} \{\cB, \cA \ast \cB, \ldots, \cA^{\star m-1} \ast \cB\}  \subseteq \R^{\JK},
\end{equation} 
where $\cA^{\star k} =  \underbrace{\cA \ast \ldots \ast \cA}_{k}$. We restrict our attention in the case where $N=M=2$ for the sake of simplicity. In other words, $4th$-order tensors will be taken into consideration in the following sections. The results can be easily  extended to the cases where $N\geq 2$ and $M \geq 2$. 
%\begin{remark}
In the sequel, unless otherwise specified, we always refer to the Einstein product between two $4th$-order tensors $"\ast_2"$ as $"\ast"$.\\
%\end{remark}
The tensor global Anroldi process is an analogue to the one described for matrices \cite{global}. The algorithm linked to it is outlined below, and it takes a form that differs from the one described in  \cite{AlaaGuide}.
\begin{algorithm}[H]
	\caption{Tensor global Arnoldi algorithm}
	\begin{enumerate}
		\item Input :  $\cA \in \R^{\JJF}$, $\cB \in \R^{\JKF}$ and a fixed integer  $m$.
		\item Output : $\ctV_{m+1} \in \R^{J_1 \times J_2 \times K_1 \times (m+1)\, K_2}$ and $\overline{H}_m \in \R^{m+1 \times m}$.
		\item Set $\beta =\|\cB\|_F$ and $\cV_1=\cB/ \beta$. \quad ({\it MATLAB notation \,\,}$\ctV_{m+1}(:,:,:,1:K_2)=\cV_1$)
		\item for $j=1 \cdots m$
		\item $\cW= \cA \ast\cV_j$
		\item \quad for $i=1,\ldots, j$.
		\item  \qquad $h_{i,j} = <\cV_i, \cW>$.
		\item  \qquad $\cW = \cW-h_{i,j} \cV_i$.
		\item \quad end for.
		\item $h_{j+1,j}= \|\cW\|_F$. If $h_{j+1,j}=0$, stop. Otherwise $\cV_{j+1}= \cW/ h_{j+1}.$
		\item [] ({\it MATLAB notation.\,\,}$\ctV_{m+1}(:,:,:,jK_2+1:(j+1)K_2)=\cV_{j+1}$).
		\item end for.
	\end{enumerate}
	\label{tengloalgo}
\end{algorithm}
\noindent
What we end up with, after $m$ steps, is the following decomposition
\begin{equation}
	\label{decomp}
	\cA \ast \ctV_m = \ctV_{m+1} \ast (\overline{H}_m \circledast \cI_K),
\end{equation}
where $\ctV_{m+1} \in \R^{\JKmm}$ contains $\cV_i \in \R^{\JKF}$ for $i=1,\ldots,m+1$, as it is explained in the algorithm above. $\overline{H}_m \in \R^{m+1 \times m}$ is a Hessenberg matrix of the following form 
$$\overline{H}_m=\begin{bmatrix}
	h_{1,1}	&h_{1,2}  & \ldots & h_{1,m}  \\
	h_{2,1}	& h_{2,2}  & \ldots & h_{2,m}  \\
	0	& h_{3,2} & \ldots  & h_{3,m}  \\
	\vdots	& \ddots &  \ddots & \vdots  \\
	0	& \ddots &  & h_{m,m}  \\
	0 &\ldots& 0 & h_{m+1,m}
\end{bmatrix}.$$
$\cI_K \in \R^{\KK}$ is the identity tensor, a diagonal tensor with entries
\begin{equation}
	\label{identen}
	\cI_K(i,j,k,l)=\delta_{i,k} \delta_{j,l}, \quad \text{where} \quad \delta_{p,q} = 1, \,\, \text{if} \,\,p=q,\,\, \delta_{p,q} = 0,\,\, \text{otherwise}. 
\end{equation} 
The product $\circledast$ used in the decompositions is defined as follows.
\medskip
\begin{df}
	\label{npd}
	Given a matrix $P\in \R^{m \times n}$ and a tensor $\cJ \in \R^{\KK}$, the resulted tensor $\cR = P \circledast \cJ $ is defined as follows
	\begin{enumerate}
		\item $\cR \in \R^{K_1 \times mK_2 \times K1 \times nK_2}$.
		\item $\cR(i,:,i,:) = P \otimes \cJ(i,:,i,:)$, for $i=1,\ldots,K_1$, where $\otimes$ is the Kronecker product.
	\end{enumerate}
\end{df}
\medskip
\noindent Notice  that if $\cJ$ is a matrix in $\R^{K_1 \times K_2}$ then the product $\circledast$ becomes the Kronecker product $\otimes$.
%\end{remark}
The following result is used in the computational process described in the next sections. It is summarized in the following straightforward proposition. 
\begin{proposition}
	Let $\ctV_m \in \R^{\JKm}$ be the tensor generated from Algorithm \ref{tengloalgo} Let $P,Q^T$ two matrices in $\R^{m\times n}$, $\cI_K \in \R^{\KK}$ the identity tensor. Then we have  
	$$\ctV_m \ast(PQ \circledast \cI_K)=\ctV_m \ast((P \circledast \cI_K) \ast(Q \circledast \cI_K)),$$
\end{proposition}
\noindent
Notice that the computed $\cV_i$'s from Algorithm \ref{tengloalgo} form an orthonormal basis of the tensor Krylov subspace \ref{t_ks}; i.e. $\ctV_m^T \diamond \ctV_m = I_m$, where $I_m \in \R^{m\times m}$ is the identity matrix. The product $\diamond$ is an analogous to the one defined in the global Arnoldi process described for the matrix case in \cite{global} as follows
\begin{equation}
	\ctV_m^T \diamond \ctV_m =\begin{bmatrix}
		<\cV_1,\cV_1>	& <\cV_1,\cV_2>  &\ldots  & <\cV_1,\cV_m> \\
		<\cV_2,\cV_1>	& <\cV_2,\cV_2> & \ldots &<\cV_2,\cV_m>  \\
		\vdots	& \ddots & \ddots &\ldots  \\
		<\cV_m,\cV_1>& \ldots &\ldots  & <\cV_m,\cV_m> 
	\end{bmatrix} \in \R^{m\times m},
\end{equation}
where $<\cdot,\cdot>$ is the inner-product defined in the previous section for tensors. Thanks to \cite{global2} which concerns some properties of the $\diamond$ product in the matrix case, we can obtain  similar results in the tensor case 
\begin{proposition}
	\label{ten_prop}
	Given tensors $\cE,\cF,\cG \in \R^{\JKm}$ and matrix $S \in \R^{m\times m}$. Then we have
	\begin{enumerate}
		\item $(\cE+\cF)^T \diamond \cG = \cE^T \diamond \cG + \cF^T \diamond \cG.$
		\item $\cE^T \diamond (\cF +\cG) = \cE^T \diamond \cF + \cE^T \diamond \cG$.
		\item $\cE^T \diamond (\cF\ast(S \circledast I_K)) = (\cE^T \diamond \cF)S$.
	\end{enumerate}
\end{proposition}
\noindent Based on these  results, and using the definition of the product $\circledast$, we can easily prove the following proposition that will be of a great use in terms of simplification as we will show in the following sections. 
Let  $\ctV_{m+1}$ be the orthonormal tensor given by  Algorithm \ref{tengloalgo}, then as for the matrix case, we have the classical algebraic relations 
\begin{equation}
	\ctV_{m+1}^T \diamond (\cA \ast \ctV_m) =  \overline{H}_m, \quad \text{and} \quad \ctV_m^T \diamond (\cA \ast \ctV_m) =  H_m. 
\end{equation}
where $H_m \in \R^{m\times m}$ is also a Hessenberg matrix extracted from $\overline{H}_m$ by deleting the last row.
\subsubsection{The extended Krylov subspace process}
For the matrix case,  extended block and global Arnoldi methods are more effective than  the classical ones. This is usually the case In model order reduction techniques \cite{Abidi,houda1}, for  large scale dynamical systems;  \cite{hey,heyouni0,simoncini,simon}. In this section,  we describe the extended tensor global Arnoldi process. Before that let us give some definitions and remarks about the extended tensor Krylov subspace. Let  $\cA$ and $\cV$ be two  tensors of appropriate dimensions, then the $2m$-dimension	 extended tensor global Krylov subspace denoted by $\cK^e_m(\cA,\cV)$, can be defined as follows
\begin{equation}
	\label{tenextkrysub}
	\begin{array}{lll}
		\cK^e_m(\cA,\cV) &=& \cK_m(\cA,\cV) + \cK_m(\cA^{-1},\cA^{-1} \ast \cV)\\
		&=& \text{span}\left\{ \cA^{-\star m}\ast \cV,\ldots,\cA^{-1}\ast\cV, \cV, \ldots,\cA^{\star \,m-1}\ast \cV\right\},
	\end{array}
\end{equation}
where $\cA^{\star k} =  \underbrace{\cA \ast \ldots \ast \cA}_{k}$, $\cA^{\star -k} =  \underbrace{\cA^{-1} \ast \ldots \ast \cA^{-1}}_{k}$ and $\cK_m(\cA,\cV)$ is the tensor classic global Krylov subspace defined in (\ref{t_ks}) associated to the pair $(\cA,\cV)$. It is worth noting that if we consider the isomorphism $\Psi$ defined previously to define the matrices $\Psi(\cA) = A \in \R^{J_1 J_2\times J_1J_2}$ and $\Psi(\cV)=V\in \R^{J_1J_2\times K_1K_2}$, then we can generalize all the results obtained in the matrix case, from the extended global Krylov subspace process applied to the pair $(A,V)$, to the tensor structure via the Einstein product. Further details are given below. Next, we  describe  the process for  the construction of the tensor $\ctV_{2m}$ associated to the extended global Krylov subspace defined above. The process can be guaranteed via the following extended Arnoldi algorithm. As mentioned earlier, our interest is focused only on $4th$-order tensors, but the following results remain true also for higher order.
\begin{algorithm}[H]
	\caption{Tensor Extended Global Arnoldi (TEGA) algorithm}
	\begin{enumerate}
		\item Input :  $\cA \in \R^{\JJF}$, $\cB \in \R^{\JKF}$ and a fixed integer  $m$.
		\item Output : $\ctV_{2(m+1)} \in \R^{J_1 \times J_2 \times K_1 \times 2(m+1)K_2}$ and $\overline{H}_m \in \R^{2(m+1) \times 2m}$.
		\item Compute the global QR-decomposition of $|\cB, \, \cA^{-1}\ast \cB|_2 \in \R^{J_1\times J_2\times K_1\times 2K_2}$ {\it i.e.,} $\text{gQR}(|\cB, \, \cA^{-1}\ast \cB|_2)= [\cV_1, \omega]$. 
		({\it MATLAB notation \,\,}$\ctV_{2(m+1)}(:,:,:,1:2K_2)=\cV_1$, $\omega =\begin{bmatrix}
			\omega_{1,1}&\omega_{1,2} \\
			0 & \omega_{2,2} 
		\end{bmatrix} \in \R^{2 \times 2}$)
		\item for $j=1 \cdots m$
		\item Set $\cW= |\cA \ast\cV_j^1, \cA^{-1}\ast \cV_j^2|_2$, where $\left\{\begin{array}{lll}
			\cV_j^1&=&\cV_j(:,:,:,(j-1)K_2+1:jK_2), \\
			\cV_j^2&=&\cV_j(:,:,:,jK_2+1:(j+1)K_2),
		\end{array}\right.$
		\item \quad for $i=1,\ldots, j$.
		\item  \qquad $H_{i,j} = \cV_i^T \diamond \cW \in \R^{2 \times 2}$.
		\item  \qquad $\cW = \cW-\cV_i\ast(H_{i,j}\circledast \cI_K)$.
		\item \quad end for.
		\item Compute the global QR (gQR) of $\cW$, {\it i.e.,} $\text{gQR}(\cW)= [\cV_{j+1}, H_{j+1,j}]$
		%\item []  \textcolor{red}{\bf ({\it MATLAB notif.\,\,}$\ctV_{2(m+1)}(:,:,:,(j+1)K_2+1:(j+3)K_2)=\cV_{j+1}$).????}
		\item end for.
	\end{enumerate}
	\label{tenextgloalgo}
\end{algorithm}
\noindent Algorithm  \ref{tenextgloalgo} computes an orthonormal tensor basis corresponding to  $\ctV_{2(m+1)}\in \R^{J_1 \times J_2 \times K_1 \times 2(m+1)K_2}$, associated to the extended global Krylov subspace (\ref{tenextkrysub}) applied to the pair $(\cA,\cB)$.  This  means that for $i,j=1,\ldots,m$, the following is verified
\begin{equation}
	\label{orthcond}
	\cV_i^T \diamond \cV_j= 0_{2\times 2} \quad i \neq j \quad \text{and} \quad \cV_i^T \diamond \cV_i = I_{2\times 2}. 
\end{equation}
An upper Hessenber matrix $\overline{H}_m \in \R^{2(m+1) \times 2m}$ is also  generated with upper block matrices $H_{i,j} \in \R^{2 \times 2}$. In Step $3$ and Step $10$ of Algorithm \ref{tenextgloalgo}, we compute a global QR decomposition to get the element $\cV_j$ of the tensor basis $\ctV_{2(m+1)}$. This decomposition is analogue to the one described in \cite{Abidi,houda1} for the extended global Krylov subspace in the matrix case. We use this decomposition to get  $|\cB, \, \cA^{-1}\ast \cB|_2= \cV_1 \ast (\omega \circledast \cI_K)$ in Step $3$ and $\cW=\cV_{j+1}\ast (H_{j+1,j}\circledast \cI_K)$ in Step $10$, where product $\circledast$ is the one described in  Definition \ref{npd}. From steps $5$ to $9$, we obtain the following relations 
\begin{equation}
	\begin{array}{lll}
		\cV_{j+1}\ast (H_{j+1,j}\circledast I_K)&=& |\cA \ast\cV_j^1, \cA^{-1}\ast \cV_j^2|_2 - \displaystyle \sum_{i=1}^{j} \cV_i\ast(T_{i,j}\circledast \cI_K),\\
		H_{j+1,j}&=&\cV_{j+1}^T \diamond |\cA \ast\cV_j^1, \cA^{-1}\ast \cV_j^2|_2\\
		&=&\cV_{j+1}^T \diamond \displaystyle \sum_{i=1}^{j+1} \cV_i\ast(H_{i,j}\circledast \cI_K).
	\end{array}
\end{equation}
After $m$ steps of Algorithm \ref{tenextgloalgo} we can show that 
\begin{equation}
	\begin{array}{lll}
		\cA \ast \ctV_{2m} &=& \ctV_{2(m+1)} \ast (\overline{T}_m E_m^T\circledast \cI_K)\\
		&=& \ctV_{2m} \ast(T_m \circledast \cI_k) + \cV_{m+1} \ast(T_{m+1,m}E_m^T\circledast \cI_K),
	\end{array}
\end{equation}
where $E_m$ is the last $2m\times 2$ block column of the identity matrix $I_m \in \R^{2m\times 2m}$,  $\overline{T}_m = \ctV_{2(m+1)}^T \diamond (\cA\ast \ctV_{2m}) \in \R^{2(m+1) \times 2m}$ and $T_m$ is a $2m \times 2m$ block Hessenberg matrix
defined by $T_m= \ctV_{2m}^T \diamond (\cA\ast \ctV_{2m}) \in \R^{2m \times 2m}$, where $T_{i,j} = \cV_i^T\diamond (\cA \ast \cV_j)$, for $i,j=1,\ldots,m$. We summarize in the following result a recursion to compute $T_m$ by avoiding additional tensor products with $\cA$.
\begin{proposition}(\cite{heyouninew})
	\label{Tm_ext}
	Let $\overline{H}_m = [h_{:,1}, \ldots,h_{:,m}] \in \R^{2(m+1) \times 2m}$ be the upper Hessenberg matrix generated from Algorithm \ref{tenextgloalgo} where $h_{:,i} \in \R^{2(m+1)\times 1}$ is the $ith$-column. A simple computation of the upper Hessenberg matrix $\overline{T}_m=[t_{:,1}, \ldots,t_{:,m}]\in \R^{2(m+1) \times 2m}$ is as follows.\\ The odd columns of the $\overline{T}_m$ are such that 
	$$t_{:,2j-1} = h_{:,2j-1}, \quad \text{for} \,\, j=1,\ldots,m,$$
	and the even column are described as
	\begin{itemize}
		\item[-] $t_{:,2} = \frac{1}{\omega_{22}}(\omega_{11}e_1^{2(m+1)}-\omega_{12}h_{:,1}),$
		\item [-] $t_{:,2j+2} = \frac{1}{h_{2j+2,2j}}(e_{2j}^{2(m+1)}-t_{:,1:2j+1}h_{1:2j+1,2j}) \quad \text{for} \,\, j=1,\ldots,m-1$,
	\end{itemize} 
	where $e_i^{(k)}$ is the $ith $ column vector of the identity matrix $I_k$ and $\omega_{1,1}, \omega_{1,2}$ and $\omega_{2,2}$ are as defined in step $2$ from Algorithm \ref{tenextgloalgo}.
\end{proposition}
\section{Order reduction via projection}
\label{secreduproj}
Based on the tensor classic and  extended Krylov subspace processes described above, we are now able to provide a reduction process to get a reduced MLTI system to the original one (\ref{oryml}).
We recall our original MLTI system 
\begin{equation}
	\left\{\begin{split} 
		\cX_{k+1} &= \cA\ast\cX_k+ \cB \ast \cU_k, \quad \cX_0=0,\\ 
		\cY_k &= \cC \ast \cX_k,
	\end{split} \right.
\end{equation}
where the tensors $\cA \in \R^{\JJF}$ and  $\cB, \cC^T \in \R^{\JKF}$. For simplicity, we consider product $\ast$ as the Einstein product $\ast_2$ between $4th$-order tensors. The associated transfer function mentioned previously in Proposition \ref{TFten} is given as
\begin{equation}
	\cF(s) = \cC \ast (s\cI-\cA)^{-1} \ast \cB.
\end{equation}
We seek to construct the following reduced MLTI system
\begin{equation}
	\label{reduMLTI}
	\left\{\begin{split} 
		\hat{\cX}_{k+1} &= \hat{\cA}\ast\hat{\cX}_k+ \hat{\cB} \ast \hat{\cU}_k, \quad \cX_0=0,\\ 
		\hat{\cY}_k &= \hat{\cC} \ast \hat{\cX}_k,
	\end{split} \right.
\end{equation}
by using Proposition \ref{TFten}, we can describe the associated transfer function as follows
$$\hat{\cF}(s) = \hat{\cC} \ast (s\hat{\cI}-\hat{\cA})^{-1} \ast \hat{\cB}.$$
In order to measure the accuracy of the resulting reduced system, one can compute the error-norm $\|\cF -\hat{\cF}\|$   with respect to a certain norm. Similarly to reduction process described in for the matrix case, an analogue process using a tensorial structure could be given. We first consider one of the Krylov subspace defined in the previous section, the tensor classic or extended  Krylov subspaces, where we project the initial MLTI system. The two subspaces are of $m$-dimension and $2m$-dimension respectively. We consider tensor $\cV_m$ as the corresponding basis tensor to one of the tensor Krylov subspace. After approximating the full order state $\cX_k$ by $\cV_m \ast \hat{\cX}_k$, and analogously to the matrix case, by using the Petrove-Galerkin condition technique, we end up with the following 
\begin{equation}
	\left\{\begin{split} 
		\cV_m^T\ast (\cV_m\ast\hat{\cX}_{k+1} &-\cA\ast\cV_m\ast \hat{\cX}_k - \cB \ast \cU_k)=0,\\ 
		\cY_k &= \cC \ast \cV_m \ast \hat{\cX}_k,
	\end{split} \right.
\end{equation}
finally, we obtain the desired reduced MLTI system (\ref{reduMLTI}), with the following tensorial structure
\begin{equation}
	\label{structtensredu}
	\hat{\cA} = \cV_m^T\ast(\cA \ast\cV_m), \quad \hat{\cB} = \cV_m^T\ast\cB, \quad \hat{\cC} =\cC \ast\cV_m.
\end{equation}
In the following we give explicitly the expression of the tensorial structure (\ref{structtensredu}) by means of the two processes based on the two Krylov subspaces already mentioned. To reach the desired goal, we need to find an approximation $\hat{\cF}$ to $\cF$, and this comes back to find an approximation $\cX_m$ to $\cX$ that verifies the multi-linear system (\ref{mlsys}) in the tensor classic Krylov subspace, {\it i.e.,} $\cX_m \in \cK_m(\cA, \cB)$,  or in the tensor extended Krylov subspace {\it i.e.,} $\cX_m \in \cK^e_m(\cA, \cB)$. The process to find such approximation based on a tensor classic Krylov subspace  (\ref{t_ks}) is described as follows. 
\begin{equation}
	\cX_m = \ctV_m \ast (y_m \circledast \cI_K) \in \R^{\JKF},
\end{equation} 
\begin{equation}
	\label{t_resi}
	\cR_m = \cB -(s\cI-\cA) \ast \cX_m \, \perp \, \cK_m(\cA, \cB),
\end{equation} 
%$$\cX_m = \bV_m \times_{(M+N+1)} y_m,$$
where $y_m$ is a vector with $m$-components and  $\ctV_m$ is the orthonormal basis associated to the tensor Krylov subspace (\ref{t_ks}). We know from the decomposition (\ref{decomp}) that
$$\cA \ast \ctV_m = \ctV_m \ast (H_m \circledast \cI_K) + h_{m+1,m}\cV_{m+1}\ast (e_m^T \circledast \cI_K),$$
and by using the fact $\cB=\|\cB\| \cV_1$ and Proposition \ref{ten_prop}, the equality (\ref{t_resi}) could be simplified as follows 
%\begin{equation*}
%		\bV_m \boxtimes^{(M+N+1)} (B-(s\cI-\cA)\ast (\bV_m \times_{(M+N+1)} y_m) )  = 0, 
%\end{equation*}
\begin{equation*}
	\begin{split}
		\ctV_m^T \diamond (\cB-(s\cI-\cA)\ast (\ctV_m \ast (y_m \circledast \cI_K))  &= 0, \\
		\ctV_m^T \diamond \cB &= \ctV_m^T \diamond ((s\ctV_m -\cA\ast \ctV_m)\ast(y_m \circledast \cI_K)), \\
		\|\cB\| e_1^m&= (sI_m-H_m)y_m , \\
		y_m&=(sI_m-H_m)^{-1}\|\cB\|e_1^m.
		%		\|B\|_F e_1 &= sy_m -H_my_m, \,\,\, \text{(using Proposition (\ref{t_prop1}))} \\
		%		y_m&= (sI_m-Hm)^{-1} \|B\|_F e_1.
	\end{split}
\end{equation*}
\noindent
Now, we can have an explicit expression of the approximation of $\cF(s)$ denoted by $\cF_m(s)$ and described as 
\begin{equation}
	\begin{split}
		\cF(s) \approx \hat{\cF}(s) = \cC \ast \cX_m &= \cC \ast(\ctV_m \ast (y_m \circledast \cI_K)) \\
		&=  \cC \ast(\ctV_m \ast (((sI_m-H_m)^{-1}\|\cB\|e_1^m) \circledast \cI_K)) \\
		&= \cC \ast\ctV_m \ast (s\cI_K^m-(H_m\circledast \cI_K))^{-1}\ast(\|\cB\|e_1^m\circledast \cI_K),
	\end{split}
\end{equation}
where $\cI_K^m \in \R^{\KKm}$ is the identity tensor. Now, from the approximation $\hat{\cF}(s)$, we can conclude the reduced MLTI system 
\begin{equation}
	\label{redmld}
	\left\{\begin{split} 
		\hat{\cX}_{k+1} &= \hat{\cA}\ast\hat{\cX}_k+ \hat{\cB} \ast \cU_k,\\% \quad (\text{discrete case})\\ 
		\hat{\cY}_k &= \hat{\cC} \ast \hat{\cX}_k, 
	\end{split} \right.
\end{equation}
with the associated tensor structure
\begin{equation}
	\hat{\cA}=H_m\circledast \cI_K, \quad \hat{\cB}=\|\cB\|e_1^m\circledast \cI_K, \quad \hat{\cC}= \cC \ast\ctV_m. 
\end{equation}
In order to quantify how well the approximated transfer function $\hat{\cF}(s)$ is closer to the original one $\cF(s)$, we propose in the following a simplified expression to the error-norm $\|\cF-\hat{\cF}\|$.
\begin{proposition}
	Let $\cF(s)$ and $\hat{\cF}(s)$ be the two transfer functions associated to the original MLTI (\ref{oryml}) and the reduced one (\ref{redmld}) respectively. Based on the previous results, we get
	\begin{equation}
		\|\cF(s)-\hat{\cF}(s)\| \leq \|\cC\ast \ctV_m \ast(s\cI-\cA)^{-1}\| \ast \|h_{m+1,m}((e_m^T(sI_m-H_m)^{-1}) \circledast \cI_K)\ast \cB_m\|.
	\end{equation}
\end{proposition}
\begin{proof}
	The error $\cF(s)-\hat{\cF}(s)$ could be expressed as
	\begin{equation*}
		\begin{split}
			\cF(s)-\hat{\cF}(s) &=\cC\ast(s\cI-\cA)^{-1} \ast \cB- \cC_m \ast (s\cI_K^m-\cA_m)^{-1}\ast\cB_m \\
			&= \cC\ast(s\cI-\cA)^{-1}\| \ast \left[\cB-(s\cI-\cA) \ast \ctV_m\ast (s\cI_K^m-\cA_m)^{-1}\ast\cB_m\right],
			%&= \cC\ast(s\cI-\cA)^{-1} \ast \left[\cB-\ctV_m\ast B_m 
			%+h_{m+1,m} \cV_{m+1} \ast(e_m^T\circledast \cI_K)\ast (s\cI_K^m-\cA_m)^{-1}\ast\cB_m\right]
		\end{split}
	\end{equation*}
	by using the decomposition (\ref{decomp}), we obtain
	\begin{equation*}
		\begin{split}
			\cF(s)-\hat{\cF}(s)&= \cC\ast(s\cI-\cA)^{-1} \ast \left[\cB-\underbrace{\ctV_m\ast B_m}_{\cB} 
			+h_{m+1,m} \cV_{m+1} \ast(e_m^T\circledast \cI_K)\ast (s\cI_K^m-\cA_m)^{-1}\ast\cB_m\right] \\
			&= \cC\ast(s\cI-\cA)^{-1} \ast \left[h_{m+1,m}\cV_{m+1} \ast((e_m^T(sI_m-H_m)^{-1}) \circledast \cI_K)\ast \cB_m\right],
		\end{split}
	\end{equation*}
	this concludes the proof.
\end{proof}
If we consider the tensor extended global Krylov subspace $\cK^e_m(\cA, \cB)$ (\ref{tenextkrysub}), we  can get the approximated transfer function $\hat{\cF}$ by following the same process described above based on the tensor classic global Krylov subspace, then, the approximated $\hat{\cF}$ transfer function has the following form
\begin{equation}
	\label{tfredumlti_ext}
	\hat{\cF}(s)=\hat{\cC} \ast (s\cI_{2m}-\hat{\cA})^{-1}\ast\hat{\cB},
\end{equation}
with the associated structure matrices
\begin{equation}
	\label{tenstruredu}
	\hat{\cA}=T_m\circledast \cI_K, \quad \hat{\cB}=(\ctV_{2m}^T\diamond \cB)\circledast \cI_K, \quad \hat{\cC}= \cC \ast\ctV_{2m}. 
\end{equation}
From Algorithm \ref{tenextgloalgo}, we know that 
\begin{itemize}
	\item $|\cB, \, \cA^{-1}\ast \cB|_2= \cV_1 \ast (\omega \circledast \cI_K)$ where $\omega \in \R^{2\times 2}$ is an upper triangular matrix from Algorithm \ref{tenextgloalgo}, then,  $$\cB_m=(\omega_{1,1}e_1^{(2m)})\circledast \cI_K,$$ where $e_1^{(2m)}$ is the first column of the identity matrix $I_{2m}$.\\
	\item $T_m= \ctV_{2m}^T \diamond (\cA\ast \ctV_{2m})$ a $2m\times 2m$ is a Hessenberg matrix that can be computed in a efficient way using Proposition \ref{Tm_ext} and $\ctV_{2m}$ is the tensor  computed by  Algorithm \ref{tenextgloalgo}.
\end{itemize}
\begin{remark}
	It is worth noting that in the case of a continuous-time MLTI system described as 
	\begin{equation}
		\left\{\begin{split} 
			\dot{\cX}(t) &= \cA\ast\cX(t)+ \cB \ast \cU(t), \quad \cX_0=0,\\ 
			\cY(t) &= \cC \ast \cX(t), 
		\end{split} \right.
	\end{equation}
	the process described above for constructing a discrete-time reduced MLTI system can be also applied  to the continuous-time MLTI system, we end up by the following reduced system
	\begin{equation}
		%\label{redmlc}
		\left\{\begin{split} 
			\dot{\cX_m}(t) &= \cA_m\ast\cX_m(t)+ \cB_m \ast \cU_t,\\ 
			\cY_m(t) &= \cC_m \ast \cX_m(t), 
		\end{split} \right.
	\end{equation}
	where the coefficient tensorial are as follows
	\begin{equation}
		\cA_m = \cV_m^T\ast(\cA \ast\cV_m), \quad \cB_m = \cV_m^T\ast\cB, \quad \cC_m =\cC \ast\cV_m.
	\end{equation}
\end{remark}
\section{Tensor Balanced Truncation}
\label{BTsec}
Besides Krylov subspace model reduction methods for large scale LTI systems, another class of methods is the known Balanced Truncation method (BT) introduced by Moore \cite{moore}. The main challenge in the BT is to solve large-scale discrete-time Lyapunov equations in order to obtain the system Gramians that will be used to generate a reduced model. In this section, we provide with a tensor Balanced Truncation method. The process of this method is as follows. First, we need to solve two discrete Lyapunov equations 
\begin{equation}
	\label{lyapequaBT}
	\cP - \cA \ast \cP \ast \cA^T= \cB \ast \cB^T, \qquad \cQ - \cA^T \ast \cQ \ast \cA= \cC^T \ast \cC, 	
\end{equation} 
where $\cA \in \R^{\JJF}$ and $\cB , \cC^T  \in \R^{\JKF} $. As mentioned before, $\cP$ and $\cQ$ are known as the reachability and  the observability Gramians. It is mentioned previously, that $\cP$ and $\cQ$ are weakly symmetric positive-definite square tensor, then we can obtain the Cholesky-like factors of the two gramians described as follows
\begin{equation}
	\cP = \cZ \ast \cZ^T, \qquad \cQ=  \cW \ast \cW^T.
\end{equation}
Where the tensors $\cZ$ and $\cW$ are of appropriate dimension presented in low-rank form. The next step consists in  computing  the   singular value decomposition (SVD) of  $\Psi(\cW^T \ast \cZ) =U\Sigma V^T$, where $\Psi$ is the isomorphism defined in (\ref{transphi}). A reasonable computation is required behind this step since tensor $\cW$ and $\cZ$ are not of large dimensions. The matrix $\Sigma$ is a diagonal matrix that contains the singular values known also as the {\it Hankel singular values} of the associated MLTI system \cite{CAN3}. This decomposition could be expressed as follows
\begin{equation}
	\begin{split}
		\Psi(\cW^T \ast \cZ) &=U\Sigma V^T \\
		&= \begin{bmatrix}
			U^1\,\, U^2
		\end{bmatrix}\begin{bmatrix}
			\Sigma^1	&   \\
			& \Sigma^2 
		\end{bmatrix} \begin{bmatrix}
			(V^1)^T \\
			(V^2)^T
		\end{bmatrix},
	\end{split}
\end{equation}
where $\Sigma_1\in \R^{r^2 \times r^2}$ and $\Sigma_2\in \R^{(N^2-r^2) \times (N^2-r^2)}$. As described in \cite{Antoulas3,stykel,moore}, a truncation step could be established, and by truncating the states that  correspond to small Hankel singular values in $\Sigma_2$. The tensorial structure of the reduced MLTI system is given as follows
\begin{equation}
	\cA_r= \cY_r^T \ast (\cA \ast \cX_r), \quad \cB_r = \cY_r^T \ast \cB, \quad \cC_r= \cC \ast \cX_r,  
\end{equation} 
with
\begin{equation}
	\cY_r= \cW_2 \ast \Psi^{-1}(U^1)\ast \Psi^{-1}(\Sigma_1)^{(-1/2)}, \qquad \cX_r= \cZ_1 \ast \Psi^{-1}(V^1)\ast \Psi^{-1}(\Sigma_1)^{(-1/2)}, 
\end{equation}
where $\Psi^{(-1)}$ is the inverse of the isomorphism defined in (\ref{transphi}). Regarding the solutions $\cP$ and $\cQ$ of the two Lyapunov equations (\ref{lyapequaBT}), we suggest a solution in a factored form via the tensor extended block and global Krylov subspace projection methods. These two methods are the focus of the upcoming section.
\section{Tensor disctete-time Lyapunov equations}
\label{disclyap}
In this section we discuss the process to get approximate solutions to the tensor discrete Lyapunov equations of the form
\begin{equation}
	\label{lyaptenequa}
	\cX -\cA\ast \cX \ast \cA^T = \cB \ast \cB^T,
\end{equation}% even for an order higher than four
where $\cA ,\cB$ are  tensors of appropriate dimensions and $\cX$ is the unknown tensor. The solution of such equation is  the main task in Balanced Truncation model order reduction method for MLTI systems described in \cite{CAN3}. They also play a crucial role in control theory, for instance, the equation (\ref{lyaptenequa}) is what we get after a discretization of the $2D$ heat equations with control \cite{CAN3,Nip}. It is obvious that, if the  dimension of (\ref{lyaptenequa}) is small, then one can transform (\ref{lyaptenequa}) to a matrix  Lyapunov equation and using some efficient direct   techniques  described in \cite{Barraud,heyouni0,simon}. But in the case of large scale tensors, choosing a process based only on tensors will  be of benefits than using some matricization techniques that will cost us both computationally and in terms of memory.  In the next section, we propose a  process to solve discrete Lyapunov tensor equations by means of tensor Krylov subspace techniques. We use the global process described previously which is based on a tensor classic or extended Krylov subspace. Moreover, a block process by means of the tensor extended block Krylov subspace for solving (\ref{lyaptenequa}) is also presented. A comparison between all the methods is assessed in the numerical section.
\subsection{Tensor global Krylov method for discrete-Lyapunov equation}
We recall the tensor Lyapunov equation that we are interested in
\begin{equation}
	\label{lyaptenequa2}
	\cX -\cA\ast \cX \ast \cA^T = \cB \ast \cB^T,
\end{equation}
where $\cA \in \R^{\JJF}$, $\cB \in \R^{\KK}$ and $\cX \in \R^{\JJF}$ is the unknown tensor. Analogously to the matrix discrete-Lyapunov equation \cite{Barraud}, we state that (\ref{lyaptenequa2}) has a unique solution if the following is satisfied
$$\lambda_{j_1j_2j_1j_2} \, \lambda_{j_1j_2j_1j_2} \neq 1, \quad j_1=1,\ldots,J_1; \,\,j_2=1,\ldots,J_2 $$
where $\lambda_{j_1j_2j_1j_2}$ are the eigenvalues of $\cA$.
\subsubsection{The extended process}
In the following we describe the case of constructing the approximation by using a projection onto the tensor extended global Krylov subspace defined in (\ref{tenextkrysub}). The approximate solution is given by

\begin{equation}
	\label{tenapproext}
	\cX_m = \ctV_{2m} \ast (Y_{2m} \circledast \cI_K) \ast \ctV_{2m}^T,
\end{equation} 
where $\ctV_{2m}$ is  a tensor obtained  after running  Algorithm \ref{tenextgloalgo} to the pair ($\cA,\cB$),  $\cI_{2m} \in \R^{K_2\times 2mK_2\times K_1\times 2mK_2}$ is the identity tensor and the product $\circledast$ is as defined in Definition \ref{npd}. $Y_{2m}\in \R^{2m\times 2m}$ is a small matrix obtained via the following Galerkin orthogonality condition 
\begin{equation}
	\label{approcond}
	\ctV_{2m}^T \diamond(\cR_m \diamond \ctV_{2m}) = 0,
\end{equation}
where $\cR_m$ is the tensor  residual   given by $\cR_m = \cX_m - \cA \ast \cX_m\ast \cA^T - \cB \ast \cB^T$. By developing the condition described in (\ref{approcond}), we get the low order  discrete Lyapunov matrix equation verified by $Y_{2m}$, which can be expressed as follows
\begin{equation}
	\label{lowdimsyl_ten_ext}
	Y_{2m}-T_m Y_{2m} T_m^T=B_mB_m^T,
\end{equation}
where $T_m\in \R^{2m\times 2m}$ is a Hessenberg matrix obtained from Algorithm \ref{tenextgloalgo} applied to the pair ($\cA,\cB$) and  $B_m = \ctV_{2m}^T \diamond \cB$. A  simplification to $B_m$ goes as 
\begin{equation*}
	B_m=\ctV_{2m}^T \diamond \cB = \omega_{11}e_1^{(2m)}, \quad \text{since} \,\, \cB=\omega_{11} \cV_1^1 \, \, (\text{from Algorithm \ref{tenextgloalgo}})
\end{equation*}
where $\cV_1^1$ is  defined using the  MATLAB notation as $\cV_1^1=\cV_1(:,:,:,1:K_2)$, $\omega_{11}$ is a scalar defined in Step $2$ of  Algorithm \ref{tenextgloalgo} and $e_1^{(2m)}$ is the first column of the identity matrix $I_{2m} \in \R^{2m \times 2m}$. The following result shows how to compute the residual in an efficient way.
\begin{theorem1}
	\label{theo_tensor_ext}
	Assume  that $m$ iterations of Algorithm \ref{tenextgloalgo} have been run. Then 
	\begin{equation}
		\|\cR_m\| \leq r_m,		
	\end{equation}
	where \begin{equation*}%(h_{m+1,m})^2\|e_m^TY_mH_m^T\|_2^2 + 
		r_m = \sqrt{2 \|T_mY_{2m}E_{(2m)}^mT_{m+1,m}^T\|_F^2 +\|(Y_{2m}E_{(2m)}^mT_{m+1,m}^T)^TY_{2m}E_{(2m)}^mT_{m+1,m}^T\|_F^2}.
	\end{equation*}
	$E_{(2m)}^m$ is the last $2m \times 2$ column block of the identity matrix $I_{2m}$.
\end{theorem1}
\begin{proof}
	By using the definition of the approximation $\cX_m$ and the decompositions that we get from Algorithm \ref{tenextgloalgo}, we end up with the following
	\begin{equation}
		\begin{split}
			\cR_m &= \cX_m - \cA \ast \cX_m \ast \cA^T - \cB \ast \cB^T\\
			&= \ctV_{2m} \ast (Y_{2m} \circledast \cI_{2m}) \ast \ctV_{2m}^T - \underbrace{\cA \ast (\ctV_{2m} \ast (Y_{2m} \circledast \cI_{2m}) \ast \ctV_{2m}^T) \ast \cA^T}_{(1)} - \underbrace{\cB\ast\cB^T}_{(2)}.
		\end{split}
	\end{equation}
	The term $(1)$ could be simplified as follows
	{\small
		$$
		\begin{array}{lll}
			\cA \ast (\ctV_{2m} \ast (Y_{2m} \circledast \cI_K) \ast \ctV_{2m}^T) \ast \cA^T &=& (\ctV_{2m} \ast (T_mY_{2m} \circledast \cI_K) + \cV_{2(m+1)}\ast (T_{m+1,m}(E_{(2m)}^m)^TY_{2m} \circledast \cI_K)) \\ &\ast& (  (T_m^T \circledast \cI_K) \ast \ctV_{2m}^T +  (E_{(2m)}^mT_{m+1,m}^T \circledast \cI_K)\ast \cV_{2(m+1)}^T)\\
			&=&\ctV_{2m} \ast (T_mY_{2m}T_m^T \circledast \cI_K) \ast \ctV_{2m}^T\\
			&+& \ctV_{2m} \ast (T_mY_{2m}E_{(2m)}^mT_{m+1,m}^T\circledast \cI_K)\ast \cV_{2(m+1)}^T \\
			&+&\cV_{2(m+1)}\ast (T_{m+1,m}(E_{(2m)}^m)^TY_{2m}T_m^T \circledast \cI_K)\ast \ctV_{2m}^T \\
			&+& \cV_{2(m+1)}\ast (T_{m+1,m}(E_{(2m)}^m)^TY_{2m}E_{(2m)}^mT_{m+1,m}^T \circledast \cI_K)\ast\cV_{2(m+1)}^T.
		\end{array}
		$$}
	From Algorithm \ref{tenextgloalgo}, we know that $\cB=\omega_{1,1} \cV_1^1$, where $\cV_1^1$ is  defined using the  MATLAB notation as $\cV_1^1=\cV_1(:,:,:,1:K_2)$ and $\omega_{1,1}$ is a scalar defined in Step $2$ of  Algorithm \ref{tenextgloalgo},  then the term $(2)$ could be written as
	$$\cB \ast \cB^T = (\omega_{1,1})^2 \ctV_{2m} \ast (e_1^{(2m)} (e_1^{(2m)})^T \circledast \cI_K) \ast\ctV_{2m}^T.$$
	Gathering all these results, and by considering $Y_{2m}-T_mY_{2m}T_m^T -\omega_{1,1}^2(e_1^{(2m)} (e_1^{(2m)})^T)=0$ mentioned in (\ref{lowdimsyl_ten_ext}), we end up with
	\begin{equation*}
		\cR_m= \ctV_{2(m+1)} \ast\left[\begin{pmatrix}
			0	& -T_mY_{2m}E_{(2m)}^mT_{m+1,m}  \\
			-T_{m+1,m}^T(E_{(2m)}^m)^TY_{2m}T_m^T	& T_{m+1,m}(E_{(2m)}^m)^TY_{2m}E_{(2m)}^mT_{m+1,m}^T 
		\end{pmatrix} \circledast \cI_K\right] \ast \ctV_{2(m+1)}^T.
	\end{equation*}
	It follows that 
	$$\|\cR_m\| \leq \sqrt{2 \|T_mY_{2m}E_{(2m)}^mT_{m+1,m}^T\|_F^2 +\|(Y_{2m}E_{(2m)}^mT_{m+1,m}^T)^TY_{2m}E_{(2m)}^mT_{m+1,m}^T\|_F^2},$$
	which ends the proof.
\end{proof}
\begin{remark}
	The  approximation $\cX_m$ could be expressed in a factored form. This allows us to compute $\cX_m$ in a efficient way. We start by employing a Singular Value Decomposition (SVD) to $Y_{2m}$, {\it i.e.,} $Y_{2m}=P\Sigma Q$, then we consider some tolerance {\tt dtol} and define $P_r$, $Q_r$ as the first $r$ columns of $P$ and $Q$ corresponding respectively to the $r$ singular values of magnitude greater than {\tt dtol}. By setting $\Sigma_r=\text{diag}(\sigma_1\ldots,\sigma_r)$, we get $Y_{2m} \approx P_r\Sigma_rQ_r^T$ and the factorization goes as 
	\begin{equation}
		\label{keyten_facto}
		\cX_m \approx \ctV_{2m} \ast (P_r\Sigma_rQ_r^T \circledast \cI_K) \ast \ctV_{2m}^T\approx \cZ_1\ast \cZ_2^T,
	\end{equation}
	where $\cZ_1 = \ctV_{2m} \ast (P_r(\Sigma_r)^{1/2} \circledast \cI_K) $ and $\cZ_2=((\Sigma_r)^{1/2}Q_r^T \circledast \cI_K) \ast \ctV_{2m}^T$.
	The same process could be followed to get the approximation in a factored form using the global extended tensor Krylov subspace, in this case $\cZ_1 = \ctV_{2m} \ast (P_r(\Sigma_r)^{1/2} \circledast \cI_K) $ and $\cZ_2=((\Sigma_r)^{1/2}Q_r^T \circledast \cI_K) \ast \ctV_{2m}^T$.
\end{remark}
All the results described in this section enable us to provide with  the following algorithm
\begin{algorithm}[H]
	\caption{ The Lyapunov tensor extended  global Arnoldi algorithm (TLTEGAA)}
	\label{TLETGAA}
	\begin{enumerate}		
		\item  Inputs: ~$\cA\in \R^{\JJF}$,  $\cB\in \R^{\JKF}$, tolerances $\epsilon$, {\tt dtol} and a number $m_{max}$ of maximum iterations.
		\item  Outputs: the approximate solution $\cX_m \approx \cZ_1 \ast\cZ_2^T$.
		\item For $m=1, \cdots, m_{max}$ 
		\item Use Algorithm \ref{tenextgloalgo} to built $\ctV_{2m}$ an orthonormal tensor associated to the tensor extended Krylov subspace (\ref{tenextkrysub}) and $H_m \in \R^{2m \times 2m}$ a Hessenberg matrix. We use Proposition \ref{Tm_ext} to compute the Hessenberg matrix $T_m\in \R^{2m \times 2m}$.
		\item Solve the low-dimensional disctete Lyapunov equation (\ref{lowdimsyl_ten_ext}) using the MATLAB function {\tt dlyap}.
		\item Compute the residual norm $\|\cR_m\|$ using Theorem \ref{theo_tensor_ext}, and if it is less than $\epsilon$, then
		\begin{enumerate}
			\item compute the SVD of $Y_{2m}= P \Sigma Q^T$ where $\Sigma= diag [\sigma_1, \cdots, \sigma_{2m}],$
			\item determine $r$ such that $\sigma_{r+1} < {\tt dtol} \leq \sigma_r$, set $\Sigma_r =diag[\sigma_1, \cdots, \sigma_r]$ and compute $\cZ_1 = \ctV_{2m} \ast (P_r(\Sigma_r)^{1/2} \circledast \cI_K)$ and $\cZ_2=((\Sigma_r)^{1/2}Q_r^T \circledast \cI_K) \ast \ctV_{2m}^T$,
		\end{enumerate} 
		end if.
		\item End For
	\end{enumerate}
	%	\label{exttenglofinal}
\end{algorithm}
\subsection{Tensor extended block Arnoldi  process for Lyapunov equations}
In this section, we define a tensor extended block Krylov subspace  via Einstein product that will be used to solve the tensor discrete Lyapunov equation (\ref{lyaptenequa}). A description of the classic block Krylov subspace methods via Einstein product has been performed in \cite{AlaaGuide}. For $\cA \in \R^{\JJ}$ and $\cB \in \R^{\JK}$, the $2m$-th block extended tensor Krylov subspace associated to the pair $(\cA, \cB)$ is defined by
\begin{equation}
	\label{blocktenext}
	\cK_m^{b-ext}(\cA,\cB) = \text{Range}\left\{\cA^{-\star m}\ast \cB,\ldots,\cA^{-1}\ast\cB, \cB, \ldots,\cA^{\star \,m-1}\ast \cB\right\},
\end{equation}
where $\cA^{\star k} =  \underbrace{\cA \ast \ldots \ast \cA}_{k}$, 
$\cA^{\star -k} =  \underbrace{\cA^{-1} \ast \ldots \ast \cA^{-1}}_{k}$ and $\cA^0=\cI$ is the identity tensor. Next,  we define the notion of upper/lower triangular tensors which will be used later on.
\medskip
\begin{df}
	\label{upperlowerten}
	Let $\cU$ and $\cL$ be two tensors in the space $\mathbb{R}^{\JI}$
	\begin{itemize}
		\item[-] $\cU$ si   an upper triangular tensor if the entries $u_{j_1,\ldots,j_N,i_1,\ldots,i_N} = 0$ when $i\text{vec}(j,J) \geq i\text{vec}(j,J)$.
		\item[-] $\cL$ si  a lower triangular tensor if the entries $r_{j_1,\ldots,j_N,i_1,\ldots,i_N} = 0$ when $i\text{vec}(j,J) \leq i\text{vec}(j,J)$,
	\end{itemize}
	where $i\text{vec}(\cdot,\cdot)$ is the mapping index mentioned in Definition \ref{transphi}.
\end{df} 
Analogously to the {\it QR decomposition} in the matrix case \cite{golub}, a similar definition to such decomposition of a  tensor $\cP$ in the space $\R^{\JK}$  is defined as follows 
%\begin{theorem1}
%	Given $\cP \in \R^{\JK}$, there exists two tensors, $\cQ \in \R^{\JK}$ and $\cR \in \R^{\KKK}$, the first one is orthogonal {\it i.e.,} $\cQ^T \ast \cQ=  \cI_K$ and the other one is an upper triangular tensor, then the factorization goes as
\begin{equation}
	\cP = \cQ \ast \cR,
\end{equation}
where $\cQ \in \R^{\JK}$  is orthogonal {\it i.e.,} $\cQ^T \ast \cQ=  \cI_K$ and  $\cR \in \R^{\KKK}$  is an upper triangular tensor (Definition \ref{upperlowerten}).
Similarly to the matrix extended block Krylov Arnoldi algorithm \cite{simon}, an analogous algorithm in a  tensor format, particularly for the $4th$-order tensor which is the main subject of this paper as mentioned previously, could be generated as follows
\begin{algorithm}[H]
	\caption{Tensor extended block Arnoldi algorithm}
	\begin{enumerate}
		\item Input :  $\cA \in \R^{\JJF}$, $\cB \in \R^{\JKF}$ and a fixed integer  $m$.
		\item Output : $\ctV_{2(m+1)} \in \R^{J_1 \times J_2 \times K_1 \times 2(m+1)K_2}$ and $\overline{H}_m \in \R^{K_1 \times 2(m+1)K_2 \times K_1 \times 2mK_2}$.
		\item Compute $QR$-decomp. to $|\cB, \, \cA^{-1}\ast \cB|_2 \in \R^{J_1\times J_2\times K_1\times 2K_2}$ {\it i.e.,} $|\cB, \, \cA^{-1}\ast \cB|_2= \cV_1 \ast \cR_0$. 
		({\it MATLAB notation \,\,}$\ctV_{2(m+1)}(:,:,:,1:2K_2)=\cV_1$)
		\item for $j=1 \cdots m$
		\item Set $\cW= |\cA \ast\cV_j^1, \cA^{-1}\ast \cV_j^2|_2$, where $\left\{\begin{array}{lll}
			\cV_j^1&=&\cV_j(:,:,:,(j-1)K_2+1:jK_2), \\
			\cV_j^2&=&\cV_j(:,:,:,jK_2+1:(j+1)K_2),
		\end{array}\right.$
		\item \quad for $i=1,\ldots, j$.
		\item  \qquad $\cH_{i,j} = \cV_i \ast \cW$.
		\item  \qquad $\cW = \cW-\cV_i \ast \cH_{i,j}$.
		\item \quad end for.
		\item Compute $QR$-decomp. to $\cW$, {\it i.e.,} $\cW= [\cV_{j+1}, \cH_{j+1,j}].$
		%	\item [] ({\it MATLAB notif.\,\,}$\ctV_{2(m+1)}(:,:,:,(j+1)K_2+1:(j+3)K_2)=\cV_{j+1}$).
		\item end for.
	\end{enumerate}
	\label{ETBAA}
\end{algorithm}
\noindent We start our process by applying Algorithm \ref{ETBAA} defined above to the pair $(\cA,\cB)$. After  $m$ iterations,   the following decomposition could be obtained 
\begin{equation}
	\label{decomblock_ext}
	\begin{array}{lll}
		\cA \ast \ctV_{2m} &= \ctV_{2(m+1)} \ast \overline{\cT}_m\\
		&= \ctV_{2m} \ast \cT_m +\cV_{2(m+1)} \ast (\cT_{m+1,m}\ast \cE_{2m}^T),
	\end{array}
\end{equation}
where $\ctV_{2m}$ is the tensor basis associated with $\cK_m^{b-ext}(\cA,\cB)$. Tensors $\overline{\cT}_m$ and $\cT_m$ are defined as follows    $$\overline{\cT}_m = \ctV_{2(m+1)}^T \ast (\cA \ast \ctV_{2m})\in \R^{K_1\times 2(m+1)K_2\times K_1\times 2mK_2}, \quad \cT_m = \ctV_{2m}^T \ast (\cA \ast \ctV_{2m})\in \R^{\KTKm}. $$ Finally the tensor $\cE_{2m}$ is get it from the identity tensor $\cI_{2m} \in \R^{\KTKm}$ as $\cE_{2m}=\cI_{2m}(:,:,:,2(m-1)K_2+1:2mK_2)$.  Analogously to the matrix case, the tensor $\cT_m$ can be defined as the restriction of the tensor $\cA$ in the extended block tensor Krylov subspace  $\cK_m^{b-ext}(\cA,\cB)$. Notice that the tensor $\overline{\cT}_m$ can be expressed as  
\begin{equation}
	\overline{\cT}_m=\begin{vmatrix}
		\cT_{1,1}	& \cT_{1,2} & \ldots  & \cT_{1,m} \\
		\cT_{2,1}	& \cT_{2,2} &\ldots  &\cT_{2,m}  \\
		\vdots	&  \ddots&\ddots  &\vdots  \\
		\cO	& \ldots &\cT_{m,m-1}  &\cT_{m,m} \\
		\cO	& \ldots & \cO &\cT_{m+1,m} 
	\end{vmatrix} \in \R^{K_1\times 2(m+1)K_2\times K_1\times 2mK_2},
\end{equation}
where the notation $|\cdot|$ is the definition of block tensors given in Section \ref{blockten}. The tensor  $\cT_m$ is obtained  from $\overline{\cT}_m$ by deleting the last row block tensor $\big\vert \cO \quad \cO \quad \ldots \quad \cH_{m+1,m}\big\vert$.
\begin{remark}
	Generally, the tensor $\overline{\cH}_m \in \R^{K_1\times 2(m+1)K_2\times K_1\times 2mK_2}$ is of small dimension, then the computation of the matricization of such tensor denoted by $mat(\overline{\cH}_m)=\overline{H}_m$ is feasible. We denote by  $\overline{T}_m$ the matrix that we get from the matricization of tensor $\overline{\cT}_m$. We recall that $\overline{T}_m$ could be computed efficiently and directly using $\overline{H}_m$ as mentioned in \cite{simoncini} without requiring any extra product. The purpose behind this process is to build $\cT_m$ needed afterwards without requiring any extra tensor product with the large scale tensor $\cA$.
\end{remark} 
We seek for an approximate solution $\cX_m$ to $\cX$ fulfilling the discrete Lyapunov equation (\ref{lyaptenequa2}). Consider the following form of the approximation
\begin{equation}
	\label{approblockext}
	\cX_m = \ctV_{2m}\ast \cY_{2m}\ast\ctV_{2m}^T,
\end{equation}
where $\ctV_{2m}$ is the tensor basis associated with the tensor Krylov subspace obtained by using Algorithm \ref{ETBAA} to the pair $(\cA,\cB)$. We obtain the tensor $\cY_{2m}\in \R^{\KTKm}$ via the following condition
\begin{equation}
	\label{approcondblext}
	\ctV_{2m}^T \ast(\cR_m \ast \ctV_{2m}) = 0,
\end{equation}
and $\cR_m$ is the residual tensor given by $\cR_m = \cX_m - \cA \ast \cX_m \ast \cA^T - \cB \ast \cB^T$. We develop the condition (\ref{approcondblext}) using the decomposition (\ref{decomblock_ext}) together with orthogonality condition ({\it i.e.,} $\ctV_{2m}^T \ast \ctV_{2m} = \cI_{2m}$), we find that the tensor $\cY_{2m}$ fulfil the following low dimensional Lyapunov tensor equation
\begin{equation}
	\label{lowdimsyltenbloext}
	\cY_m - \cT_m\ast \cY_m\ast \cT_m^T = \cB_m\ast \cB_m^T.
\end{equation}
The right hand side could be simplified to the following $\cB_m\ast \cB_m^T$. We know from Algorithm \ref{ETBAA} that 
\begin{equation*}
	\begin{split}
		|\cB, \, \cA^{-1}\ast \cB|_2&= \cV_1 \ast \cR_0 \\
		&= |\cV_1^1,\cV_1^2|_2 \ast \begin{vmatrix}
			\cR_0^{(1,1)} & \cR_0^{(1,2)} \\
			\cO & \cR_0^{(2,2)}
		\end{vmatrix}, 
	\end{split}
\end{equation*}
where $\cV_1^i = \cV_1(:,:,:,(i-1)K_2+1:iK_2) \in \R^{\KK}$, $\cR_0^{(i,j)} = \cR_0(:,(i-1)K_2+1:iK_2,:,(j-1)K_2+1:jK_2)\in \R^{\KK}$. The notation $|\cdot|_2$ is the definition of row block tensor given in Definition \ref{rowblock} and the upper triangular tensor $\cR_0$ defined in Definition \ref{upperlowerten} can be presented as described above using notation from Section \ref{blockten}. Using Proposition \ref{rowblocktimesrowblock}, we can conclude the following 
$$\cB = \cV_1^1 \ast \cR_0^{(1,1)},$$
thus
\begin{equation}
	\label{rhsext}
	\begin{split}
		\cB_m &= \ctV_{2m}^T \ast \cB\\
		&= \cE_{2m}^1 \ast \cR_0^{(1,1)},
	\end{split}
\end{equation}
where $\cE_{2m}^1$ can be extracted from the identity tensor $\cI_{2m} \in \R^{\KTKm}$ as follows
$\cE_{2m}^1= \cI_{2m}(:,:,:,1:K_2).$ In the next result, we show how to simplify the expression of the residual $\cR_m$, this yields to a proper computation of the residual tensorial norm $\|\cR_m\|$. Our process is stopped whenever the residual norm $\|\cR_m\|$ is lower than a chosen tolerance $\epsilon$ ({\it i.e.$\|\cR_m\| < \epsilon$ }).
\medskip
\begin{theorem1}
	\label{extbresinorm}
	Assume that $m$ iterations have been run using Algorithm \ref{ETBAA}, then 
	\begin{equation}
		\|\cR_m\| =r_m,		
	\end{equation}
	where \begin{equation*}
		r_m = \sqrt{2*\|\cT_m\ast \cY_{2m} \ast \cE_{2m} \ast \cT_{m+1,m}^T \|^2 +\|\cT_{m+1,m}\ast \cE_{2m}^T \ast \cY_{2m} \ast \cE_{2m}\ast \cT_{m+1,m}^T\|^2},
	\end{equation*}
	and $\|\cdot\|$ is the Frobenius tensor norm defined in Definition \ref{tennorm}.
\end{theorem1}
\begin{proof}
	We plug $\cX_m$ in the expression of $\cR_m$ and by using decomposition (\ref{decomblock_ext}) and the fact that $\cY_{2m}-\cT_m\ast \cY_{2m}\ast \cT_m^T -\cE_{2m}^1 \ast \cR_0^{(1,1)} \ast (\cE_{2m}^1 \ast \cR_0^{(1,1)})^T =0$ ({\it i.e.,}\, (\ref{lowdimsyltenbloext})), we get 
	{%\footnotesize
		\begin{equation*}
			\cR_m= \ctV_{2(m+1)} \ast\left[\begin{vmatrix}
				0	& -\cT_m \ast \cY_{2m} \ast \cE_{2m} \ast \cT_{m+1,m}^T  \\
				-(\cT_m \ast \cY_m \ast \cE_{2m} \ast \cT_{m+1,m}^T)^T	& \cT_{m+1,m} \ast \cE_{2m}^T \ast\cY_{2m} \ast\cE_{2m} \ast \cT_{m+1,m}^T 
			\end{vmatrix} \right] \ast \ctV_{2(m+1)}^T.
	\end{equation*}} 
	Then we obtain the  desired result 
	$$\|\cR_m\| = \sqrt{2*\|\cT_m\ast \cY_{2m} \ast \cE_{2m} \ast \cT_{m+1,m}^T \|^2 +\|\cT_{m+1,m}\ast \cE_{2m}^T \ast \cY_{2m} \ast \cE_{2m}\ast \cT_{m+1,m}^T\|^2}.$$
\end{proof}

Thanks to all the results described above, we are now able to present the following algorithm
\begin{algorithm}[H]
	\caption{ The Lyapunov  tensor extended block  Arnoldi algorithm (TLTEBAA)}
	\label{TLBETAA}
	\begin{enumerate}		
		\item  Inputs: ~$\cA\in \R^{\JJF}$,  $\cB\in \R^{\JKF}$, tolerances $\epsilon$, {\tt dtol} and a number $m_{max}$ of maximum iterations.
		\item  Outputs: the approximate solution $\cX_m \approx \cZ_1 \ast \cZ_2^T$.
		\item For $m=1, \cdots, m_{max}$ 
		\item Use Algorithm \ref{ETBAA} to built $\ctV_{2m}$ an orthonormal tensor associated to the tensor Krylov subspace $\cK_m^{b-ext}(\cA,\cB)$ and $\cH_m \in \R^{\KTKm}$ a type Hessenberg form tensor.
		\item Solving the low-dimensional Lyapunov tensor equation (\ref{lowdimsyltenbloext}) after matricization using the MATLAB function {\tt dlyap}.
		\item Compute the residual norm $\|\cR_m^e\|$ using Theorem \ref{extbresinorm}, and if it is less than $\epsilon$, then
		\begin{enumerate}
			\item Compute the SVD of $\cY_m= \cP \ast \Sigma \ast \cQ^T$ where $\Sigma$ is a diagonal tensor contains singular values $ \left\{\sigma_1, \cdots, \sigma_{2m}\right\}.$ This definition of tensor SVD via Einstein product could be found in \cite{Brazell}.
			\item Determine $r$ such that $\sigma_{r+1} < {\tt dtol} \leq \sigma_r$, set $\Sigma_r$ as a diagonal tensor with only the first $r$ singular values, then we compute $\cZ_1 = \ctV_{2m} \ast (\cP_r\ast (\Sigma_r)^{1/2} )$ and $\cZ_2=((\Sigma_r)^{1/2}\ast \cQ_r^T) \ast \ctV_{2m}^T$. $\cP_r$ and $\cQ_r$ are defined as $\cP_r=\cP(:,:,:,1:2rK_2)$ and $\cQ_r=\cQ(:,:,:,1:2rK_2).$
		\end{enumerate} 
		end if.
		\item End For
	\end{enumerate}
	%\label{ebasaalgo}
\end{algorithm}
\begin{remark}
	Dealing with the continuous time MLTI systems (\ref{orymlc}) lead to the appearance of continuous time Lyapunov equations of the form $\cA \ast \cX - \cX \ast \cA^T = \cB \ast \cB^T$. The methods described above for finding an efficient approximation to the discrete-time Lyapunov equations (\ref{lyaptenequa}) can also be applied to the continuous-time Lyapunov equations.
\end{remark}
\section{Numerical experiments}
\label{numerexamp}
In this section, numerical examples are performed to validate the efficiency of the algorithms described above based on the tensor extended global and block Krylov subspaces. The numerical results were obtained using MATLAB R2018a on a computer with Intel $^\text{\textregistered}$ core i7 at 2.6GHz and 16Gb of RAM. We need to mention that all the algorithms described here, have been implemented based on the MATLAB tensor toolbox  established by Kolda et al., \cite{Kolda1}. \\
We start our numerical tests  by showing the results obtained by  the tensor extended global algorithm to get a reduced MLTI system from the original one. The tensorial structure of the constructed reduced MLTI system is described in (\ref{tenstruredu}). We are interested in the input-output behaviour of the reduced MLTI system. Such behaviour could be measured by the closeness of the original and reduced transfer functions, denoted by $\cF(s)$ and $\cF_m(s)$, respectively. Analogously to the matrix case (\cite{discsys}), we can define the following norm named $h_{\infty}$-norm (cf. \cite{antoulas}) to describe the magnitudes needed for plotting the two transfer functions and see how close the reduced MLTI system to the original one\\
\begin{equation}
	\label{hintennorm}
	\|\cF(.)\|_{\infty} = \|\Psi(\cF(.))\|_{\infty} =\sup_{\theta\in [0,2\pi]} \sigma_{max}(\Psi(\cF(e^{i\theta}))).
\end{equation}
{\bf Example 1.} We consider the following data 
\begin{itemize}
	\item [-] $\cA \in \R^{N\times N \times N \times N}$,  with  $N=100$. Here, $\cA$ is constructed from a tensorization of a triangular matrix $A \in \R^{N \,N \times N \,N}$ constructed using {\tt spdiags} MATLAB function.\\
	\item [-] $\cB, \cC^T \in \R^{N\times N \times K_1\times K_2}$ are chosen as a  sparse and random tensors respectively with $K_1=3$ and $K_2=5$.
\end{itemize}
The choice of $m=5$ which is the dimension of the tensor extended global Krylov subspace where we project the original MLTI system was sufficient to obtain the results described in Figure \ref{fig1}.
\begin{figure}[H]
	%	\centering
	\includegraphics[width=7.7cm]{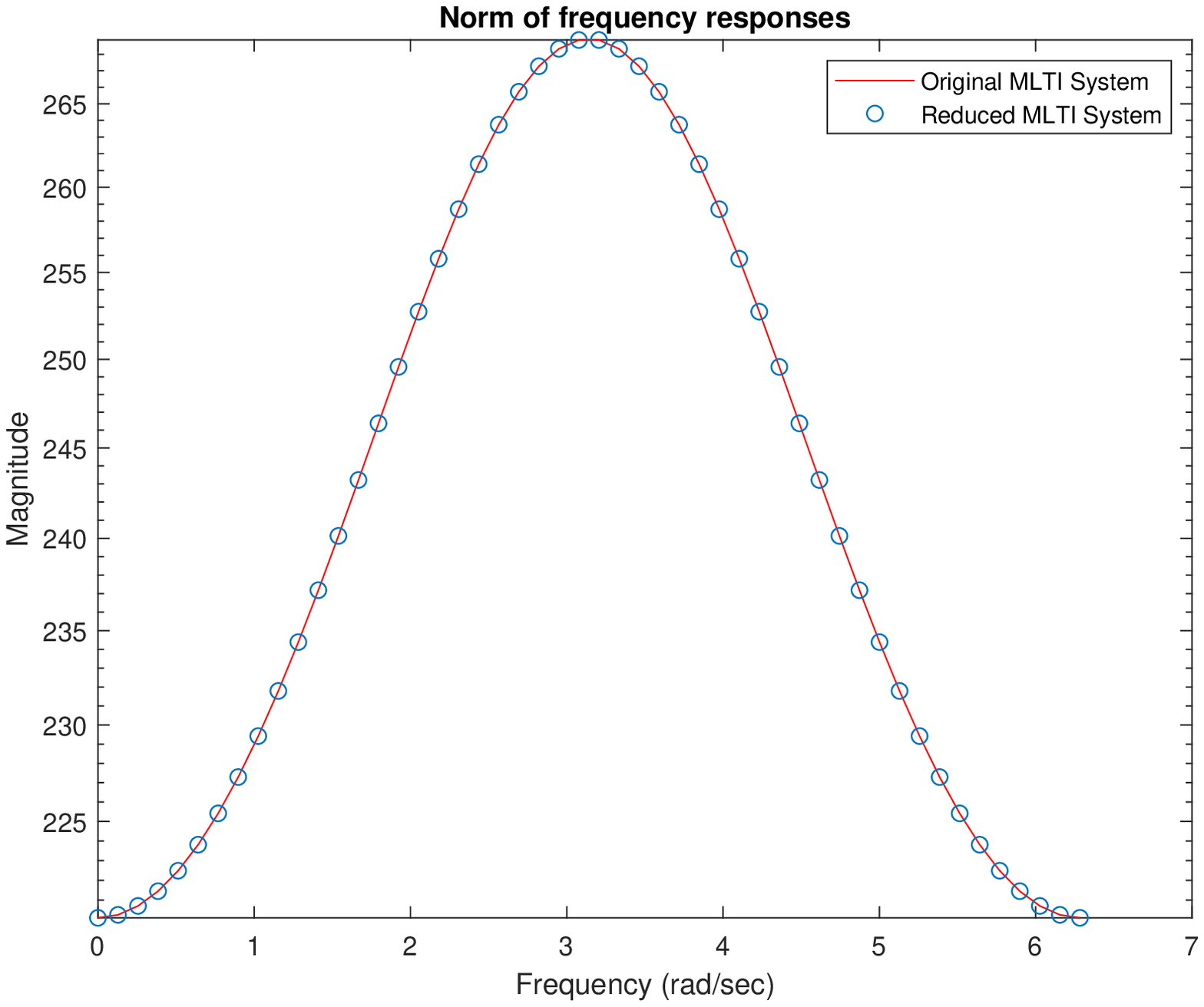}
	\hfill
	\includegraphics[width=7.7cm]{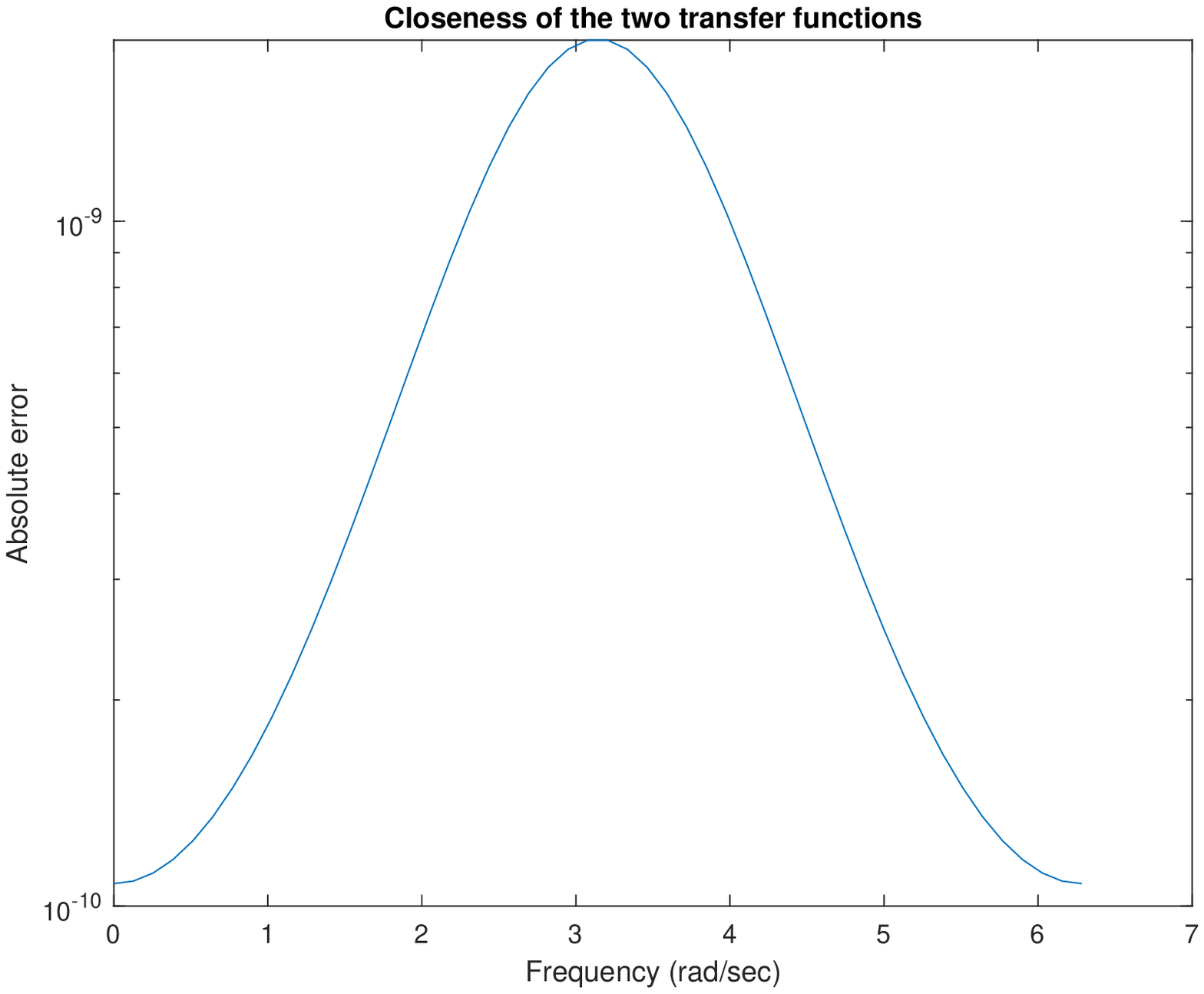}
	\caption{The frequency responses of the original and reduced MLTI systems ({\it left}) and the error norms $\|\cF-\cF_m\|_{\infty}$ ({\it right}).}\label{fig1}
\end{figure}
\noindent {\bf Example 2.} This example is from (\cite{CAN3}) and  it concerns the evolution of heat distribution in a solid medium over time. The partial differential equation that describes this evolution is given by the following equation and named the $2D$ heat equation  
\begin{equation}
	\left\{\begin{split} 
		\dfrac{\partial}{\partial t}\phi(t,x) &= c^2 \dfrac{\partial^2}{\partial x^2} \phi(t,x) +\delta(x) u_t, \quad x\in D \, (\text{square }D=[-\pi^2,\pi]),\\ 
		\phi(t,x) &= 0, \quad x\in \partial D,
	\end{split} \right. 
\end{equation}
where $c>0$, $u_t$ is a one dimensional control input and $\delta(x)$ is the Dirac delta function centred at zeros. According to \cite{CAN3}, a second-order central difference is used to approximate the Laplacien and a first-order difference in time, this leading to a MLTI system of the form (\ref{oryml}),where $\cX_k\in \R^{N\times N}$ called the 2D-temperature field at instance $k$.
For further information, we refer the readers to \cite{CAN3}, the same problem has been treated in a different way in \cite{Brazell}. We consider the following system tensors
\begin{itemize}
	\item [-] The tensor $\cA \in \R^{N \times N \times N \times N}$ with $N=128$ is the the tensorization of $\frac{c^2 \Delta t}{h^2} \Delta_{dd} \in  \R^{N^2 \times N^2}$, where $\Delta_{dd}$ is the discrete Laplacian on a rectangular grid with a Dirichlet boundary condition (see \cite{CAN3} for more details).
	\item [-] The constructed tensor $\cB$ in \cite{CAN3} is of dimension
	$N\times N \times 1\times 1$. Here, we change $\cB$ and we choose it as a  sparse tensor in $ \R^{N\times N \times K_1\times K_2}$  with $K_1=3$ and $K_2=5$.
	\item [-] To complete the system, we will need the output tensor $\cC$. We choose it as random tensor in $\R^{K_1\times K_2\times N\times N}$ with $K_1=3$ and $K_2=5$.
	\item [-] We set $m=20$, the associated dimension of the extended tensor Krylov subspace (\ref{tenextkrysub}). 
\end{itemize}
\begin{figure}[H]
	%	\centering
	\includegraphics[width=8cm]{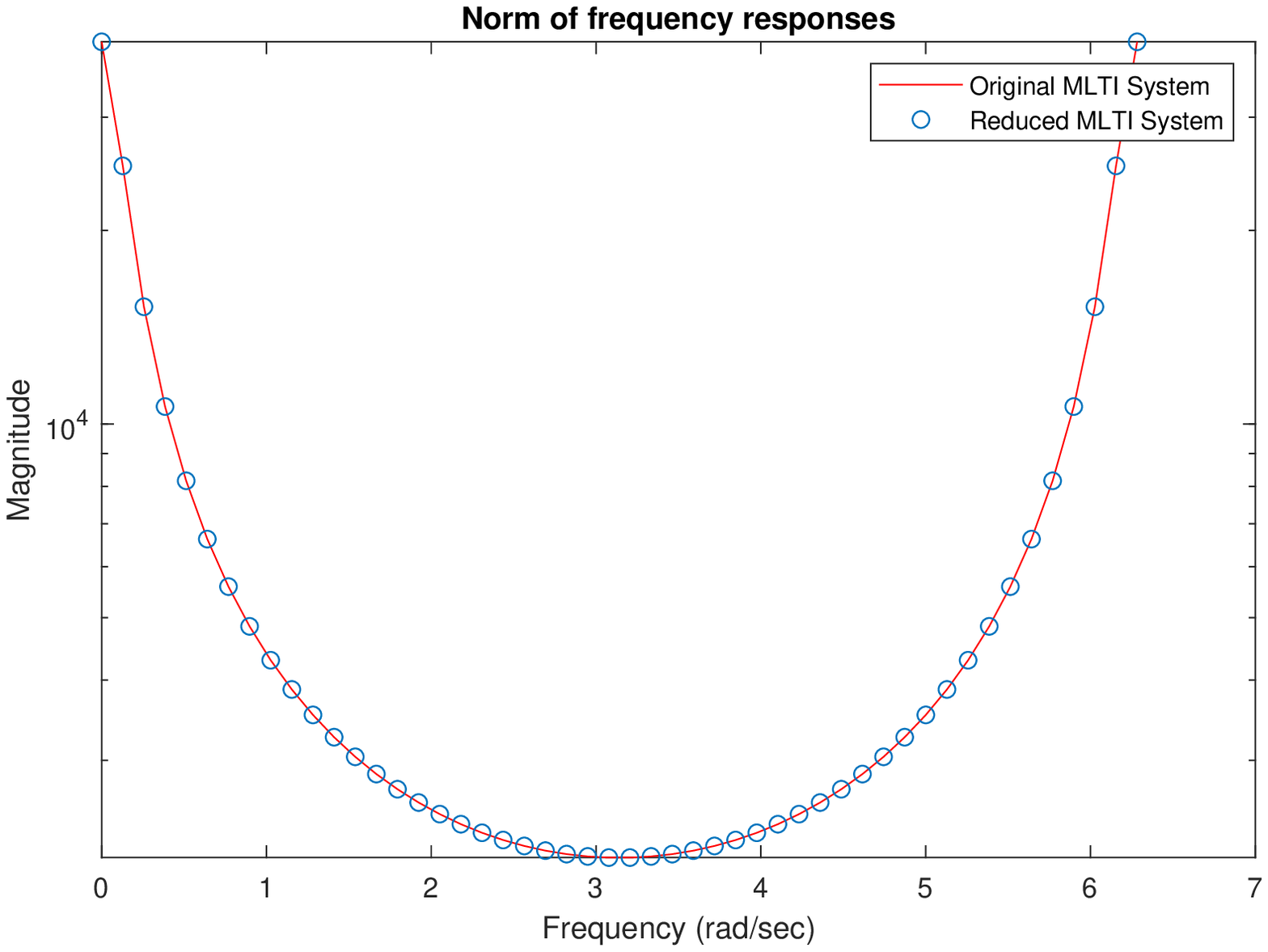}
	\hfill
	\includegraphics[width=7.2cm]{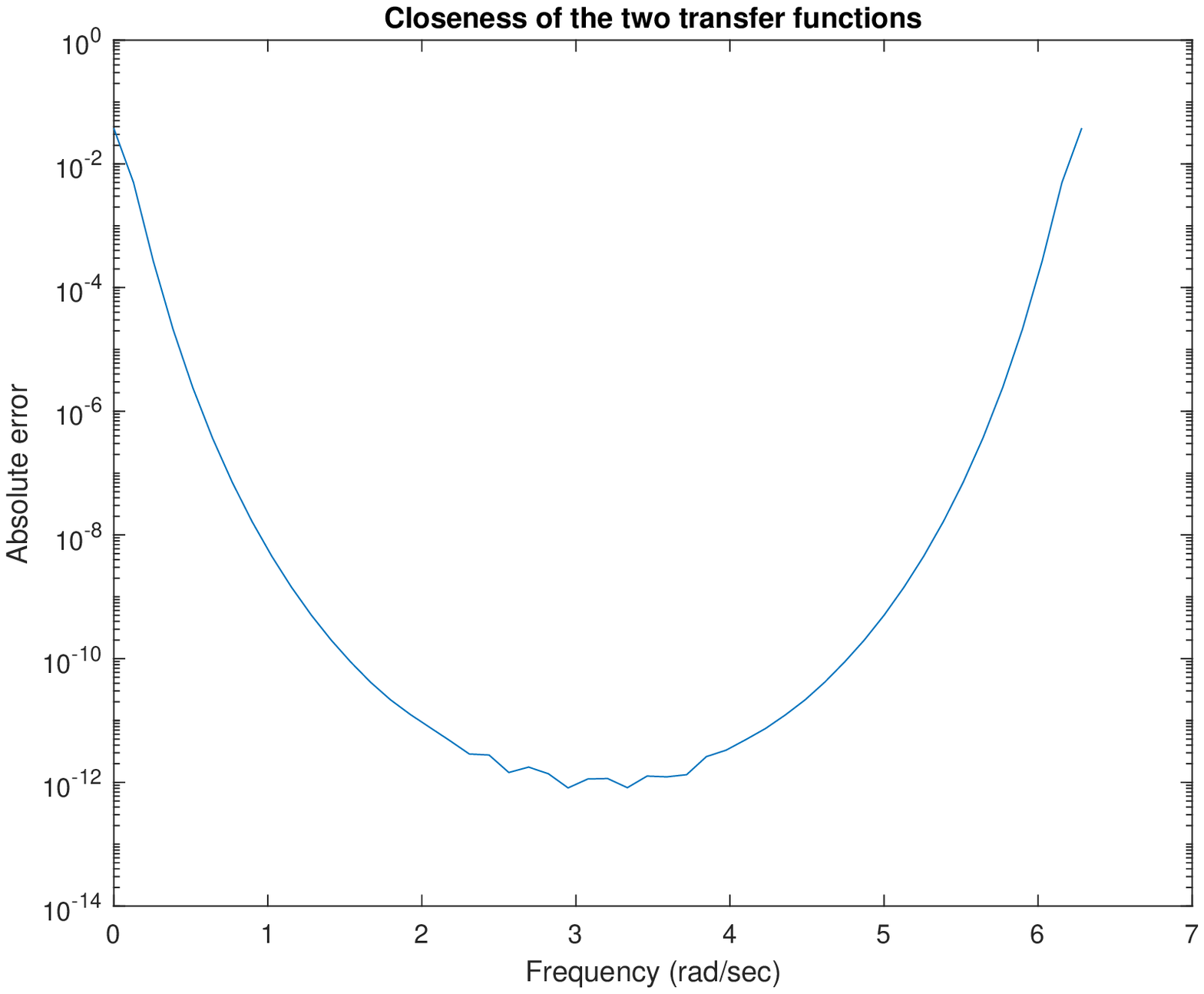}
	\caption{The frequency responses of the original and reduced MLTI systems ({\it left}) and the error norms $\|\cF-\cF_m\|_{\infty}$ ({\it right}).}
\end{figure}
{\bf Example 3.}
For this experiment, we show the results obtained from the tensor Balanced Truncation method explained in Section \ref{BTsec}, we consider the following data 
\begin{itemize}
	\item [-] $\cA \in \R^{N\times N \times N \times N}$,  with  $N=100$. Here, $\cA$ is constructed from a tensorization of a matrix $A \in \R^{N^2 \times N^2}$ constructed using the identity matrix {\tt speye} perturbed by {\tt sprandn} MATLAB function.\\
	\item [-] $\cB, \cC^T \in \R^{N\times N \times K_1\times K_2}$ are chosen as a random tensors with $K_1=5$ and $K_2=5$.
	\item[-] We use Algorithm \ref{TLETGAA} to solve the two discrete Lyapunov equations (\ref{lyapequaBT}) by setting $m=20$ (\it i.e., max number of iterations) and the tolerance $\epsilon = 1e^{-6}$.
\end{itemize}
\begin{figure}[H]
	%	\centering
	\includegraphics[width=7cm]{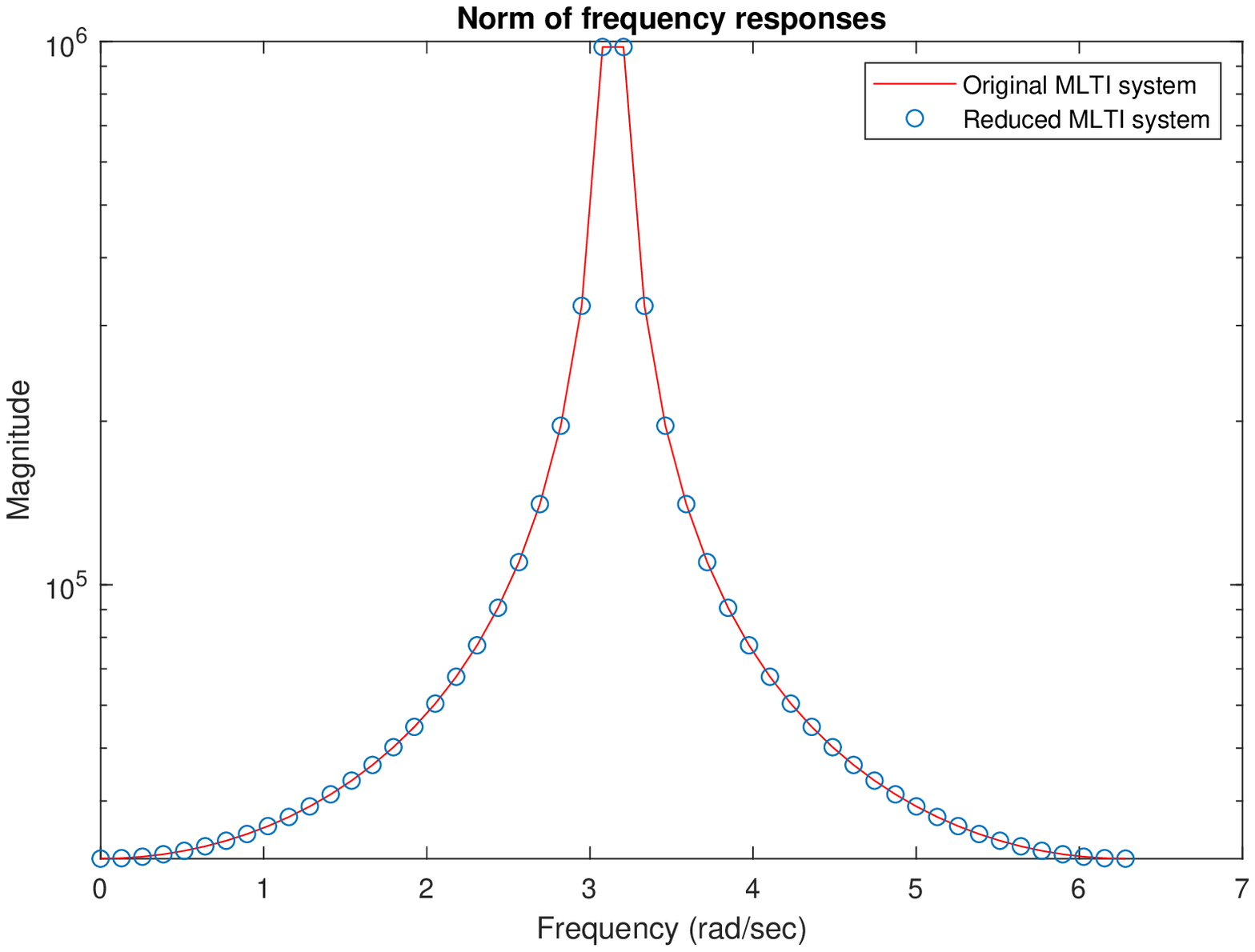}
	\hfill
	\includegraphics[width=7cm]{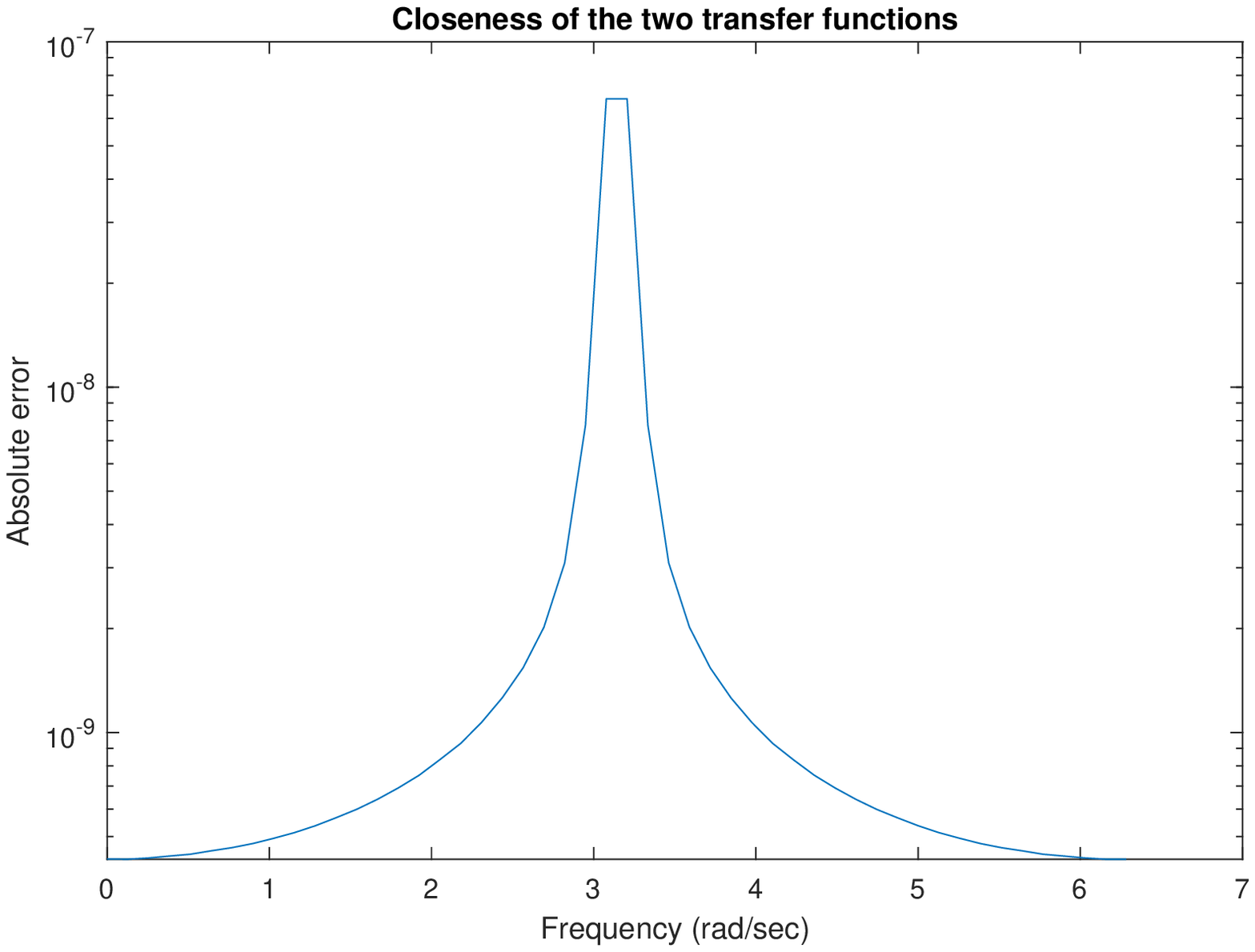}
	\caption{The frequency responses of the original and reduced MLTI systems ({\it left}) and the error norms $\|\cF-\cF_m\|_{\infty}$ ({\it right}).}
\end{figure}
For the next experiment, we show the results obtained by  the methods described  in Algorithms \ref{ETBAA} and \ref{TLBETAA}. These methods are seeking for an efficient approximate solution of  the discrete Lyapunov tensor equation $\cX -\cA\ast \cX \ast \cA^T = \cB \ast \cB^T$. We stopped our iterations whenever $\|\cR_m\|<\epsilon$, 
where $\|\cdot\|$ is the Frobenius norm associated to tensors defined in Section \ref{prelem} and  $\epsilon$ is some tolerance that will be mentioned in each example. 
We compared our methods based on the tensor extended global and block Krylov subspaces, in Algorithm \ref{TLETGAA} ({\it for the global process}) and in Algorithm \ref{TLBETAA} ({\it for the block process}) with the classic global Arnoldi described in Algorithm \ref{tengloalgo}  and block Arnoldi process described in \cite{AlaaGuide}. Consider the following data 
\begin{itemize}
	\item [-] $\cA \in \R^{\JJF}$,  with  $J_1=J_2$ in $\left\{32,64,128\right\}.$ Here, $\cA$ is a  random sparse tensor  ({\it i.e.,} MATLAB func. {\tt sptenrand}).
	\item [-] $\cB \in \R^{\JKF}$ is chosen as a sparse tensor with $K_1=3$ and $K_2=5$.
	\item [-] Th tolerance $\epsilon$ is set to $10^{-6}$.
\end{itemize}
\begin{figure}[H]
	\centering
	\includegraphics[width=7cm]{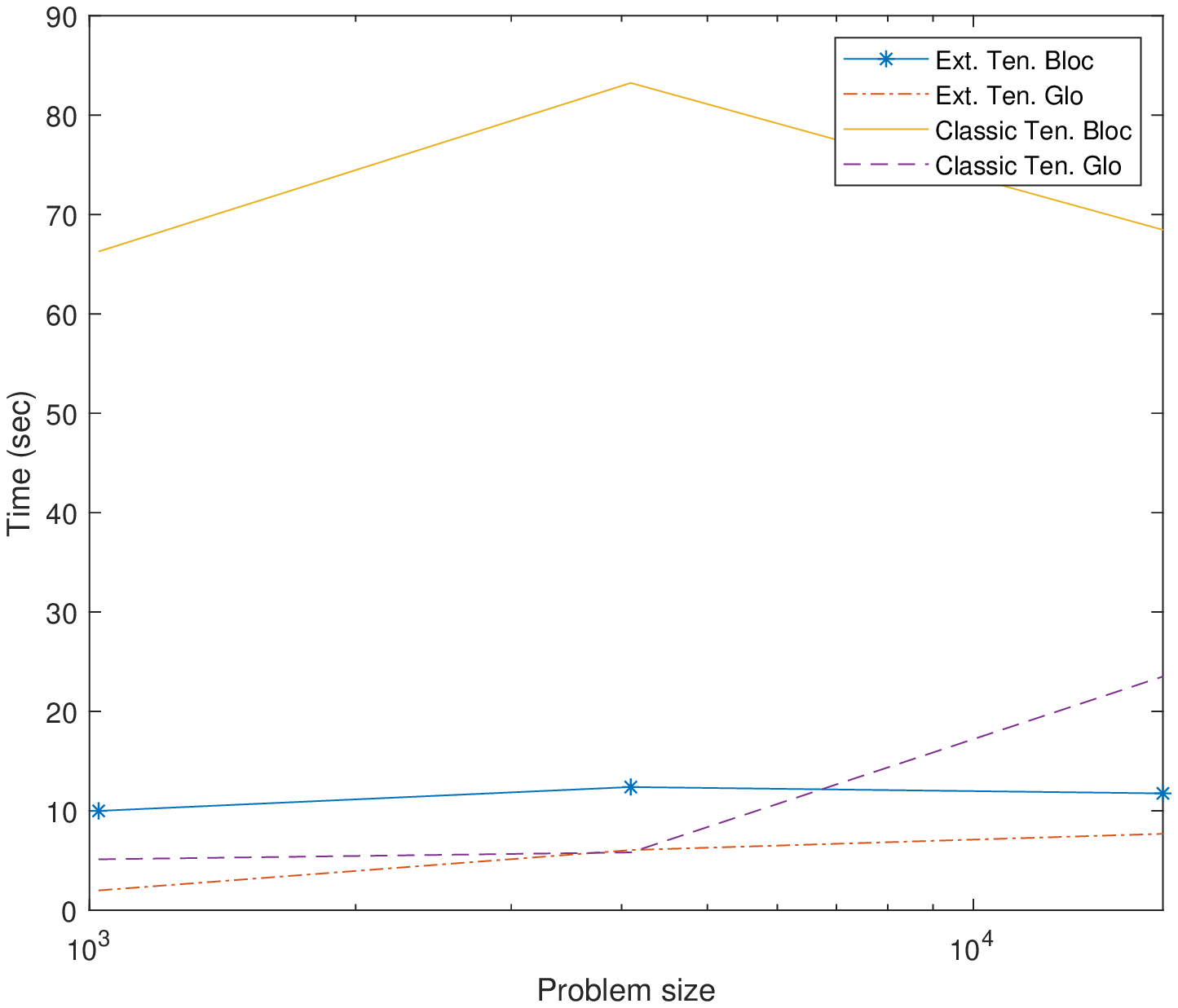}
	%\caption{Consumed time by the four methods to reach the criterion convergence.}
	\hfill
	\includegraphics[width=7cm]{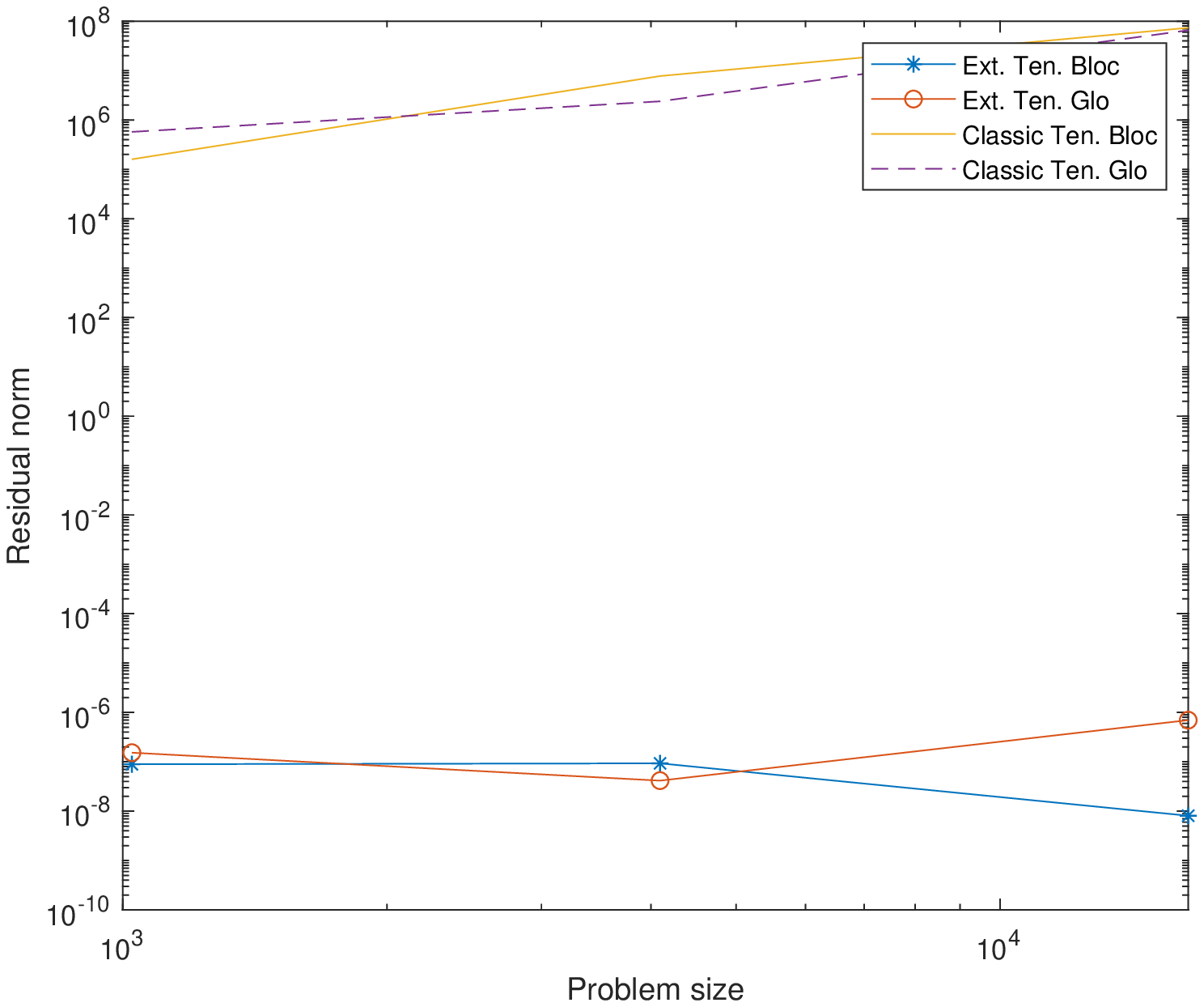}
	\caption{Consumed time by the four methods to reach the criterion convergence ({\it left}) and the residual norm ({\it right}).}
	\label{cpuerror1}
\end{figure}
In the left plot of Figure \ref{cpuerror1}, we show the cpu-time consumed by the four methods to achieve the  convergence criterion ({\it i.e.,} $\|\cR_m\|<1e^{-6}$). The dimension of the tensor extended Krylov subspace was fixed to  $m=20$, and since this latter is of $2m$ dimension, we used $m_1=2m$ the dimension of the tensor classic Krylov subspace in order to realise a fair comparison. This latter is considered for all the examples described here. The tensor extended block and global Krylov subspace methods have successfully achieved the desired result in a few seconds as mentioned in the right plot of Figure \ref{cpuerror1}, contrary to the tensor classic  block and global  Krylov subspace methods where no convergence has been guaranteed and more iterations were need.

\noindent For the  next  experiment, the tolerance was  set to  $\epsilon =10^{-8}$ and the following data has been considered 
\begin{itemize}
	\item [-] $\cA \in \R^{\JJF}$,  we set $J_1=J_2$ in $\left\{32,64,128,256\right\}.$ Here, $\cA$ is constructed from a tensorization of a triangular matrix $A \in \R^{J_1J_2 \times J_1J_2}$ constructed using {\tt spdiags} MATLAB function.
	\item [-] $\cB \in \R^{\JKF}$ is chosen as a random tensor with $K_1=3$ and $K_2=5$. In Figure, $\cB$ is a random tensor with $K_1=5$ and $K_2=6$.
\end{itemize}
\begin{figure}[H]
	\centering
	\includegraphics[width=7cm]{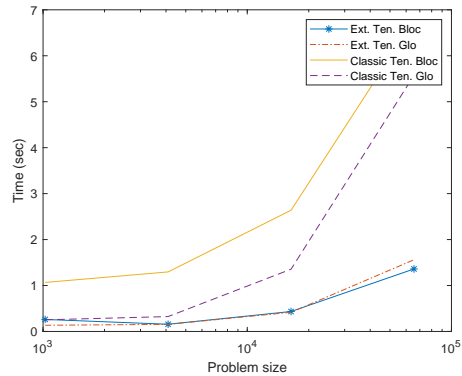}
	\caption{Consumed time by the four methods to reach the criterion convergence.}
	\label{cpu2}
\end{figure}
In Figure \ref{cpu2}, we  represent the computation time used  for the four methods to reach the convergence criterion. In Figure \ref{cpu3}, we plotted the cpu-time needed to achieve a good convergence using two examples of $\cB$; for the left plot we chose $\cB\in \R^{J_1\times J_2\times 5\times 6}$, and for the right plot, we show the results depicted from choosing  $\cB\in \R^{J_1\times J_2\times 6\times 8}$. 
\begin{figure}[H]
	%	\centering
	\includegraphics[width=7cm]{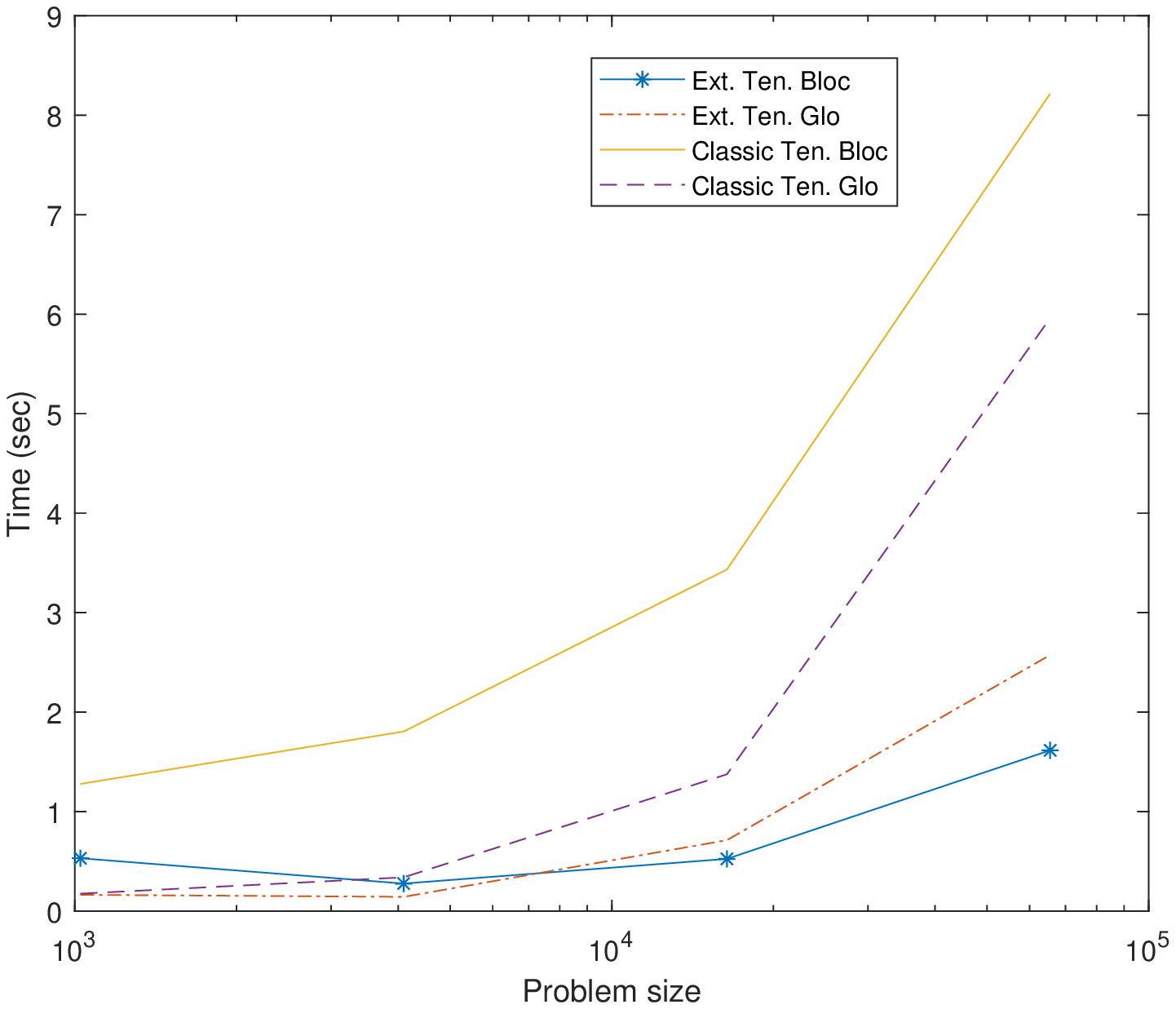}
	\hfill
	\includegraphics[width=7.3cm]{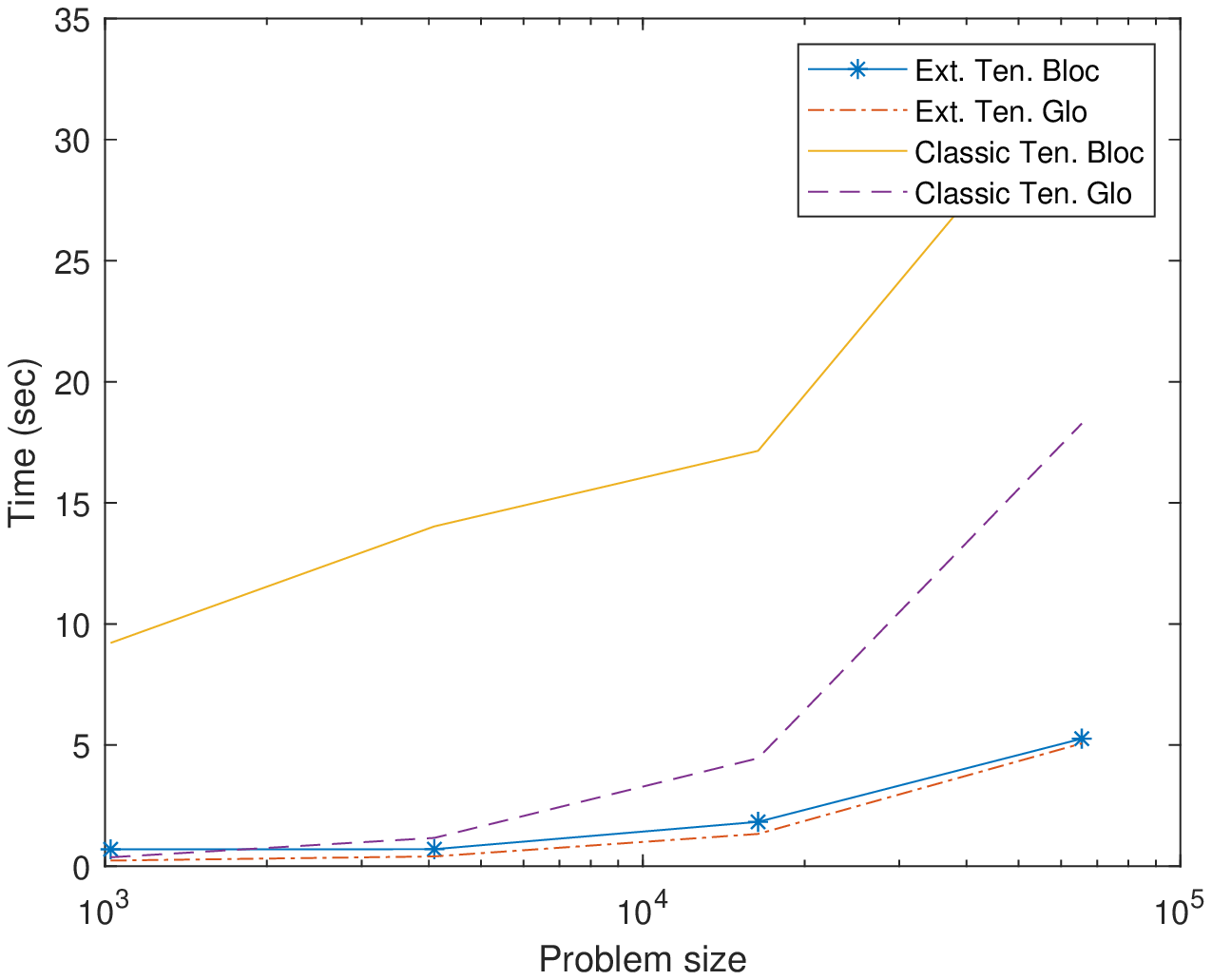}
	\caption{Consumed time by the four methods to reach the criterion convergence, where the right hand side $\cB\in \R^{J_1\times J_2\times 5\times 6}$  ({\it left}) $\cB\in \R^{J_1\times J_2\times 6\times 8}$ ({\it right}).}
	\label{cpu3}
\end{figure}
In Table \ref{table}, we present the results obtained  from  the tensor  classic and extended  Krylov subspace methods using the block and global process to solve high order discrete Lyapunov tensor equations. We stopped our iterations when  the residual norm is less than $\epsilon=10^{-8}$. We considered  the following data 
\begin{itemize}
	\item [-] $\cA \in \R^{\JJF}$,  we set $J_1=J_2$ in $\left\{512,1024\right\}.$ Here, $\cA$ is constructed from a tensorization of a triangular matrix $A \in \R^{J_1J_2 \times J_1J_2}$ constructed using {\tt spdiags} MATLAB function. After that we use {\tt spreshape} to get the tensor form.\\
	\item [-] $\cB \in \R^{\JKF}$ is chosen as a random tensor with $K_1=5$ and $K_2=6$.
\end{itemize}
\begin{table}[!h]
	\caption{Results gained from the four methods for solving high order Lyapunov tensor equations.}
	\centering %\begin{center} 
	%		\vskip0.2cm
	%		\begin{adjustbox}{tabular=lll,center}
		\resizebox{\textwidth}{!}{\begin{tabular}{l|c|c|c|c}
				\midrule
				%\\[0.1cm]
				Size of $\cA$ ($\JJF$)  & Classic Kry. Bloc & Classic Kry. Global & Extended Kry. Bloc & Extended Kry. Global    
				\\%[0.2cm] 
				\midrule
				%[0.1cm]
				& cpu-time \quad Resi-norm \quad Iter. & cpu-time \quad Resi-norm \quad Iter. &cpu-time \quad Resi-norm \quad Iter.&cpu-time \quad Resi-norm \quad Iter.	\\[0.1cm]
				\midrule %[0.05cm]
				$J_1 =J_2 = 512$  & {\bf 42.84} \quad 4.07 $\times$ 10$^{-09}$ \quad 25 & {\bf 36.25} \quad 5.41$\times$ 10$^{-09}$ \quad 26 & {\bf 9.63} \quad 2.81$\times$ 10$^{-09}$ \quad 7 & {\bf 14.37} \quad 1.83$\times$ 10$^{-10}$ \quad 8 \\[0.1cm]
				\midrule
				$J_1 =J_2 = 1024$  & {\bf 172.02} \quad 9.42 $\times$ 10$^{-09}$ \quad 26 & {\bf 149.88} \quad 6.78$\times$ 10$^{-09}$ \quad 27 & {\bf 48.91} \quad 1.75$\times$ 10$^{-10}$ \quad 8 & {\bf 57.14} \quad 7.10$\times$ 10$^{-10}$ \quad 8\\[0.1cm]
				\midrule
		\end{tabular}}\label{table}
	\end{table}
	
	\medskip 
	\section*{Conclusion}
	In this paper, we explored the application of multilinear  algebra in reducing the order of multidimentional linear  time invariant (MLTI) systems. We used tensor Krylov subspace methods as  key tools, which involve approximating the system solution within a low-dimensional subspace. We introduced  the tensor  extended block and global Krylov subspaces  and the corresponding Arnoldi based processes. Using these methods, we developed a model reduction   using  projection techniques. We have also shown how these methods could be used to solve large-scale Lyapunov tensor equations that are required in the balanced truncation method which is a technique for order reduction. We demonstrated how to extract approximate solutions  via the Einstein product using the tensor extended block Arnoldi and the extended global Arnoldi processes.
%\section*{Conclusion}
%A tensor Krylov-based method has been explored in this paper. We have started by describing the tensor classical Krylov subspace method, then a introduction to the structure tensor of the extended Krylov subspace defined for the matrix case is established afterwards. We follow an Arnoldi process based on a global and block Krylov subspace to construct the key tool of theses projection methods, which is the associated orthonormal tensor basis denoted $\ctV_m$ ({\it in the classical tensor Krylov subspace case}) or $\ctV_{2m}$ ({\it in the extended tensor Krylov subspace case}),   Application of these methods to MLTI model order reduction has been presented as well as the solution of the discrete Lyapunov tensor equation. From both methods, and based on the two processes, global and block process, we have derived an expression between the original tensor transfer function that describes the input-output behaviour of the MLTI system and its approximation. Concerning the approximate solution to some discrete Lyapunov tensor equation, some theorems have been proposed to simplify the expression of the residual which yields to a proper computation of the residual tensorial norm. This latter is used to stop the iterations of the proposed Algorithms. Numerical results were  finally reported  to confirm the performance of the proposed methods.

\bibliographystyle{siam}
\bibliography{ML_MOR_arxiv_ver}
\end{document}